\documentclass[a4paper]{amsart}
\usepackage{graphicx}
\usepackage{amssymb}
\usepackage{amsmath}
\usepackage{amsthm}
\usepackage{amscd}
\usepackage{color}
\usepackage{colordvi}
\usepackage[all,2cell]{xy}

\UseAllTwocells \SilentMatrices
\newtheorem*{theorem*}{Theorem A}
\newtheorem*{theorem**}{Theorem B}
\newtheorem{theorem}{Theorem}[section]
\newtheorem{corollary}[theorem]{Corollary}
\newtheorem*{corollary*}{Corollary}
\newtheorem{lemma}[theorem]{Lemma}
\newtheorem*{lemma*}{Lemma}
\newtheorem{proposition}[theorem]{Proposition}
\newtheorem*{proposition*}{Proposition}
\theoremstyle{remark}
\newtheorem{remark}[theorem]{Remark}
\newtheorem*{remark*}{Remark}

\newtheorem*{example*}{Example}
\newtheorem*{observation*}{Observation}
\theoremstyle{definition}

\newtheorem*{definition*}{Definition}

\newtheorem*{conjecture*}{Conjecture}

\numberwithin{equation}{section}

\begin{document}
\title[Quillen equivalence for chain homotopy categories]{Quillen equivalence for chain homotopy categories induced by balanced pairs}
\author[Jiangsheng Hu, Wei Ren, Xiaoyan Yang, Hanyang You] {Jiangsheng Hu, Wei Ren, Xiaoyan Yang, Hanyang You$^{\dag}$}

\thanks{}


\thanks{${}^{\dag}$Corresponding author: youhanyang@hznu.edu.cn}

\maketitle

\dedicatory{}%
\commby{}%
\begin{abstract}
For a balanced pair $(\mathcal{X},\mathcal{Y})$ in an abelian category, we investigate when the chain homotopy categories ${\bf K}(\mathcal{X})$ and ${\bf K}(\mathcal{Y})$ are triangulated equivalent. To this end, we realize these chain homotopy categories as homotopy
categories of certain model categories and give
conditions that ensure the existence of a Quillen equivalence between the model categories in question. We further give applications to cotorsion triples, Gorenstein projective and Gorenstein injective modules, as well as pure projective and pure injective objects.
\end{abstract}

\section{Introduction}\label{introduction}
\noindent Let ${\bf K}(\mathcal{P})$ and ${\bf K}(\mathcal{I})$ denote the chain homotopy categories of complexes of projective modules and injective modules, respectively. For a commutative Noetherian ring $R$ with a dualizing complex $D$, Iyengar and Krause established a triangle-equivalence
$D\otimes-: {\bf K}(\mathcal{P})\rightarrow {\bf K}(\mathcal{I})$; by considering the classes of compact objects in ${\bf K}(\mathcal{P})$ and ${\bf K}(\mathcal{I})$, the equivalence may be viewed as an extension of Grothendieck's duality theorem for the bounded derived category ${\bf D}^f(R)$; cf. \cite[Theorem 4.2]{IK06}. Subsequently, Chen proved that for any ring $R$ of finite Gorenstein global dimension, there is a triangle-equivalence ${\bf K}(\mathcal{GP})\simeq{\bf K}(\mathcal{GI})$ between the chain homotopy categories of complexes of Gorenstein projective and Gorenstein injective modules, which restricts to ${\bf K}(\mathcal{P})\simeq{\bf K}(\mathcal{I})$; cf. \cite[Theorem B]{Chen10}. In the case of commutative Gorenstein rings,  this recovers the Iyengar-Krause equivalence up to natural isomorphism; cf. \cite[Proposition 6.2]{Chen10}. More recently, Wang and Estrada reobtained Chen's equivalence by weakening the hypothesis to finite Gorenstein weak global dimension; cf. \cite[Theorem 1.2]{WE24}.

In fact, Chen \cite{Chen10} obtained the above triangle-equivalence in a more general setting. Specifically, any admissible balanced pair $(\mathcal{X}, \mathcal{Y})$ of finite dimension induces a triangle-equivalence ${\bf K}(\mathcal{X})\simeq{\bf K}(\mathcal{Y})$ of chain homotopy categories; cf. \cite[Theorem A]{Chen10}. Recall that a pair of additive subcategories $(\mathcal{X}, \mathcal{Y})$ in an abelian category $\mathcal{A}$ is balanced if any object in $\mathcal{A}$ admits an $\mathcal{X}$-resolution which is acyclic after applying ${\rm Hom}_\mathcal{A}(-,Y)$ for all $Y\in \mathcal{Y}$, and a $\mathcal{Y}$-coresolution which is acyclic after applying ${\rm Hom}_\mathcal{A}(X,-)$ for all $X\in \mathcal{X}$. This leads to the balanced phenomenon that the relative right-derived functors of ${\rm Hom}_\mathcal{A}(-,-)$ can be computed via an $\mathcal{X}$-resolution of the first variable, or equally via a $\mathcal{Y}$-coresolution of the second variable; that is, the functor ${\rm Hom}$ is right balanced by $\mathcal{X}\times\mathcal{Y}$; cf. \cite[\S 8.2]{EJ00}. The balanced pair $(\mathcal{X}, \mathcal{Y})$ is admissible precisely when every right $\mathcal{X}$-approximation
is epic, or equivalently, every left $\mathcal{Y}$-approximation is monic; cf. \cite[Corollary 2.3]{Chen10}. 
For instance, $(\mathcal{P},\mathcal{I})$ is trivially an admissible balanced pair, while $(\mathcal{GP},\mathcal{GI})$ is an admissible balanced pair when the base ring is virtually Gorenstein (e.g., Gorenstein ring); cf. \cite{Bel00, DLW23, EPZ20, SS20}.

It is noteworthy that there exist admissible balanced pairs of infinite dimension that induce a triangle-equivalence of the corresponding chain homotopy categories. For example, let $R$ be a product of $\aleph_{\omega}$ copies of a field. Since $R$ is von Neumann regular, every short exact sequence of left $R$-modules is pure (see \cite[Exercise 29, p.162]{Lam99}). Consequently, pure projective left $R$-modules are projective, and $R$ has infinite pure global dimension (see \cite{Os70}). By Corollary \ref{cor:KPP=KPI}, the pair $(\mathcal{PP},\mathcal{PI})$ of pure projective and pure injective left $R$-modules forms an admissible balanced pair of infinite dimension, yielding a triangulated equivalence between ${\bf K}(\mathcal{PP})$ and ${\bf K}(\mathcal{PI})$, as required. 

This motivates the search for broader sufficient conditions under which a balanced pair $(\mathcal{X},\mathcal{Y})$ yields a triangle-equivalence ${\bf K}(\mathcal{X})\simeq{\bf K}(\mathcal{Y})$. The equivalence constructed by Chen in \cite[Theorem A]{Chen10} proceeds via a comparison of the relative derived categories ${\bf D}_\mathcal{X}(\mathcal{A})$ and ${\bf D}_\mathcal{Y}(\mathcal{A})$. In contrast, we establish the desired triangle-equivalence through Quillen equivalences of the associated model categories.

Following Quillen \cite{Qui67}, a model category is a category with three distinguished classes of morphisms, called cofibrations, fibrations and weak equivalences, which satisfy a few axioms that are reminiscent of certain properties of continuous maps on topological spaces. The associated homotopy category is obtained by formally inverting the weak equivalences. For instance, ${\bf K}(\mathcal{P})$ (resp. ${\bf K}(\mathcal{I})$) is equivalent to the homotopy category of the contraderived (resp. coderived) model structure introduced by Positselski \cite{Pot11}; see also \cite{Bec14, Gil08}. Under mild conditions, 
${\bf K}(\mathcal{GP})$ and ${\bf K}(\mathcal{GI})$ admit similar model-categorical realizations; see details in \cite{Gil16}.

Now let $(\mathcal{X}, \mathcal{Y})$ be an admissible balanced pair in an abelian category $\mathcal{A}$. We denote by $\mathcal{E}$ the class of short exact sequences in $\mathcal{A}$ which remain exact under ${\rm Hom}_{\mathcal{A}}(X, -)$ for all $X\in \mathcal{X}$; then $(\mathcal{A}, \mathcal{E})$ is an exact category. Consequently, ${\rm Ch}(\mathcal{A}, \mathcal{E})$, the category of complexes with short exact sequences lying degree-wise in $\mathcal{E}$, is also exact (see \cite[Lemma 9.1]{Buh10}). Let $\mathcal{E}\text{-}{\rm dw}\widetilde{\mathcal{X}}$ be the class of complexes $X$ with each $X_n\in \mathcal{X}$, where the prefix ``$\mathcal{E}$'' indicates that orthogonality is taken with respect to ${\rm Ext}^1_{{\rm Ch}(\mathcal{E})}(-, -)$. The class $\mathcal{E}\text{-}{\rm dw}\widetilde{\mathcal{Y}}$ and its left orthogonal $^{\perp}(\mathcal{E}\text{-}{\rm dw}\widetilde{\mathcal{Y}})$ are defined analogously.

First, we realize the chain homotopy categories ${\bf K}(\mathcal{X})$ and ${\bf K}(\mathcal{Y})$ as homotopy categories of suitable model categories. Following the correspondence between model structures and Hovey triple (cf. \cite{BR07, Gil11, Hov02}), we denote a model structure $\mathcal{M}$ by its corresponding Hovey triple, and write ${\rm Ho}(\mathcal{M})$ for its homotopy category. If $(\mathcal{E}\text{-}{\rm dw}\widetilde{\mathcal{X}})^{\perp}$ is closed under direct sums, we show that 
$\mathcal{M}_{dw\mathcal{X}} = (\mathcal{E}\text{-}{\rm dw}\widetilde{\mathcal{X}}, (\mathcal{E}\text{-}{\rm dw}\widetilde{\mathcal{X}})^{\perp}, {\rm Ch}(\mathcal{A}, \mathcal{E}))$ defines a model structure on ${\rm Ch}(\mathcal{A}, \mathcal{E})$, with a triangle-equivalence ${\rm Ho}(\mathcal{M}_{dw\mathcal{X}}) \simeq {\bf K}(\mathcal{X})$ (see Theorem \ref{thm:M5}(1)). Dually, if ${}^{\perp}(\mathcal{E}\text{-}{\rm dw}\widetilde{\mathcal{Y}})$ is closed under direct products, then $\mathcal{M}_{dw\mathcal{Y}} = ({\rm Ch}(\mathcal{A}, \mathcal{E}), {^{\perp}(\mathcal{E}\text{-}{\rm dw}\widetilde{\mathcal{Y}})}, \mathcal{E}\text{-}{\rm dw}\widetilde{\mathcal{Y}})$ is a model structure on ${\rm Ch}(\mathcal{A}, \mathcal{E})$, with a triangle-equivalence
${\rm Ho}(\mathcal{M}_{dw\mathcal{Y}}) \simeq {\bf K}(\mathcal{Y})$ (see Theorem \ref{thm:M5}(2)). 

Consequently, the Quillen equivalence of model categories yields the following induced triangle-equivalences, which partially address our original question. Since $\mathcal{X}$ (resp. $\mathcal{Y}$) consists of projective (resp. injective) objects in $(\mathcal{A}, \mathcal{E})$, the model structures $\mathcal{M}_{dw\mathcal{X}}$ and $\mathcal{M}_{dw\mathcal{Y}}$ can be regarded as relative analogues of Positselski's contraderived and coderived model structures, respectively.
\vspace{0.05in}

\noindent{\bf Theorem} (see Theorem \ref{thm:KX=KY}).
Assume that $(\mathcal{E}\text{-}{\rm dw}\widetilde{\mathcal{X}})^{\perp}= {}^{\perp}(\mathcal{E}\text{-}{\rm dw}\widetilde{\mathcal{Y}})$. Then there is a triangle-equivalence ${\bf K}(\mathcal{X}) \simeq {\bf K}(\mathcal{Y})$.
\vspace{0.05in}

The theorem applies in several natural settings, showing that the condition $(\mathcal{E}\text{-}{\rm dw}\widetilde{\mathcal{X}})^{\perp}= {}^{\perp}(\mathcal{E}\text{-}{\rm dw}\widetilde{\mathcal{Y}})$ is not too restrictive. For a ring $R$, let $\mathcal{GP}$ and $\mathcal{GI}$ denote the full subcategories of Gorenstein projective and Gorenstein injective $R$-modules, respectively. If $R$ is left virtually Gorenstein and the equalities $\mathcal{E}\text{-}{\rm dg}\widetilde{\mathcal{GP}} = \mathcal{E}\text{-}{\rm dw}\widetilde{\mathcal{GP}}$ and $\mathcal{E}\text{-}{\rm dg}\widetilde{\mathcal{GI}} = \mathcal{E}\text{-}{\rm dw}\widetilde{\mathcal{GI}}$ hold (precise definitions appear in Section 5.2), then $(\mathcal{E}\text{-}{\rm dw}\widetilde{\mathcal{GP}})^{\perp}= {}^{\perp}(\mathcal{E}\text{-}{\rm dw}\widetilde{\mathcal{GI}})$. Consequently, we obtain a triangle-equivalence ${\bf K}(\mathcal{GP}) \simeq {\bf K}(\mathcal{GI})$ that restricts to ${\bf K}(\mathcal{P}) \simeq {\bf K}(\mathcal{I})$ (Corollary \ref{cor:Rmodcase}). In particular, any ring of finite Gorenstein weak global dimension is left virtually Gorenstein and satisfies the two equalities (see \cite[Theorem A]{DLW23} and Lemma \ref{lem:G-Tri2}), thereby recovering and extending the equivalences of \cite[Theorem B]{Chen10} and \cite[Theorem 1.2]{WE24} via a different approach (Corollaries \ref{cor:KGP=KGI'} and \ref{cor:KGP=KGI}).

In addition, let $\mathcal{A}$ be a locally finitely presented Grothendieck category, and let $\mathcal{PP}$ (resp. $\mathcal{PI}$) denote the class of pure projective (resp. pure injective) objects. Then $(\mathcal{PP}, \mathcal{PI})$ is an admissible balanced pair in $\mathcal{A}$, and by \cite[Theorem 5.4]{S14} we have $(\mathcal{E}\text{-}{\rm dw}\widetilde{\mathcal{PP}})^{\perp} = {^\perp}(\mathcal{E}\text{-}{\rm dw}\widetilde{\mathcal{PI}})$ on ${\rm Ch}(\mathcal{A}, \mathcal{E})$. The theorem thus yields a triangle-equivalence ${\bf K}(\mathcal{PP}) \simeq {\bf K}(\mathcal{PI})$ (Corollary \ref{cor:KPP=KPI}), which was previously obtained by \v{S}\v{t}ov\'{\i}\v{c}ek in \cite[Corollaries 5.7 and 5.8]{S14}. This also refines an observation of Chen \cite{Chen10}, where the same equivalence required the hypothesis that the ring has finite pure global dimension.
\vspace{0.1in}

\noindent{\bf Notation.} All subcategories are always considered to be full additive and closed under isomorphisms and direct
summands. Complexes are written using the homological notation, meaning that the degree decreases in
 the direction of differential maps.

\section{Preliminaries}\label{prelimi}

An {\em exact category} is a pair $(\mathcal{A}, \mathcal{E})$ where $\mathcal{A}$ is an additive category and $\mathcal{E}$
is a class of kernel-cokernel pairs $(i, p): A'\stackrel{i}\rightarrowtail A\stackrel{p}\twoheadrightarrow A''$ in $\mathcal{A}$, called short exact sequences, satisfying the standard axioms. The morphism $i$ is an {\em admissible monomorphism}, and $p$ an {\em admissible epimorphism}. The exact category $(\mathcal{A}, \mathcal{E})$ is {\em weakly idempotent complete} if every split monomorphism admits a cokernel and every split epimorphism admits a kernel; cf. \cite{Buh10, Gil11} for further details.
\vspace{0.1in}

\noindent{\bf 2.1. Cotorsion pairs in exact categories.}
Let ${\rm Ext}^1_{\mathcal{E}}(-, -)$ be the Yoneda Ext in the exact category $(\mathcal{A}, \mathcal{E})$. A pair of classes $(\mathcal{F}, \mathcal{C})$ in $(\mathcal{A}, \mathcal{E})$ is a \emph{cotorsion pair} provided that $\mathcal{F} =  {^\perp}\mathcal{C}$ and $\mathcal{C} = \mathcal{F}^{\perp}$, where $^{\perp}\mathcal{C}$ consists of objects $F$ such that $\mathrm{Ext}^1_{\mathcal{E}}(F, C) = 0$ for all $C\in \mathcal{C}$, and $\mathcal{F}^{\perp}$ consists of objects $C$ such that $\mathrm{Ext}^1_{\mathcal{E}}(F, C) = 0$ for all $F\in \mathcal{F}$. We say the cotorsion pair $(\mathcal{F}, \mathcal{C})$ is {\em hereditary} if $\mathcal{F}$ is closed under taking kernels of admissible epimorphisms between objects of $\mathcal{F}$, and if $\mathcal{C}$ is closed under taking cokernels of admissible monomorphisms between objects of $\mathcal{C}$.

Let $\mathcal{X}$ be a subcategory of $\mathcal{A}$. A morphism $f: X\rightarrow M$ with $X\in \mathcal{X}$ is called a {\em right $\mathcal{X}$-approximation} (or \emph{$\mathcal{X}$-precover}) of $M$, if any morphism from an object in $\mathcal{X}$ to $M$ factors through $f$. The subcategory $\mathcal{X}$ is called {\em contravariantly finite} (or \emph{precovering}) if each object in $\mathcal{A}$ has a right $\mathcal{X}$-approximation. Dually, for a class of objects $\mathcal{Y}$, one has the notion of a {\em left $\mathcal{Y}$-approximation} (or \emph{$\mathcal{Y}$-preenvelope}), and then the notion of a {\em covariantly finite} (or \emph{preenveloping}) subcategory.

The cotorsion pair $(\mathcal{F}, \mathcal{C})$ is said to be {\em complete} if for any object $M\in \mathcal{A}$, there exist short exact sequences $C\rightarrowtail F \twoheadrightarrow M$ and $M\rightarrowtail C' \twoheadrightarrow F'$ with $F, F'\in \mathcal{F}$ and $C, C'\in \mathcal{C}$. In this case, $F \twoheadrightarrow M$ is called a \emph{special right $\mathcal{F}$-approximation} (or \emph{special $\mathcal{F}$-precover}) of $M$, and $M\rightarrowtail C'$ is called  a \emph{special left $\mathcal{C}$-approximation} (or \emph{special $\mathcal{C}$-preenvelope}) of $M$. In this case,
$\mathcal{F}$ is contravariantly finite, and $\mathcal{C}$ is covariantly finite.
\vspace{0.1in}

\noindent{\bf 2.2. Balanced pairs.}
Throughout this subsection, let $\mathcal{A}$ be an abelian category and $\mathcal{X}$ a contravariantly finite subcategory of $\mathcal{A}$. For any object $M\in \mathcal{A}$, an {\em $\mathcal{X}$-resolution} of $M$ is a complex $\cdots \rightarrow X_{2}\stackrel{d_{2}}\rightarrow X_1\stackrel{d_{1}}\rightarrow X_0\stackrel{d_{0}}\rightarrow M\rightarrow 0$ with each $X_i\in \mathcal{X}$ such that it remains acyclic after applying ${\rm Hom}_\mathcal{A}(X,-)$ for all $X\in \mathcal{X}$; cf. \cite[Definition 8.1.2]{EJ00}. We sometimes denote the $\mathcal{X}$-resolution by $X^{\bullet}\rightarrow M$, where $X^{\bullet}=\cdots \rightarrow X_{2}\stackrel{d_{2}}\rightarrow X_1\stackrel{d_{1}}\rightarrow X_0\rightarrow 0$ is \emph{the deleted $\mathcal{X}$-resolution} of $M$.
The {\em $\mathcal{X}$-resolution dimension} of an object $M$, denoted by $\mathcal{X}\text{-}{\rm dim}M$, is defined to be the minimal integer $n\geq 0$ such that there is an $\mathcal{X}$-resolution $0\rightarrow X_n\rightarrow \cdots \rightarrow X_1\rightarrow X_0\rightarrow M\rightarrow 0$. For any subcategory $\mathcal{A}'$ of $\mathcal{A}$, the {\em global $\mathcal{X}$-resolution dimension} of $\mathcal{A}'$ is defined to be the supremum of the $\mathcal{X}$-resolution dimension of all the objects in $\mathcal{A}'$, and is denoted by  $\mathcal{X}\text{-}{\rm dim}\mathcal{A}'$.
Dually, for any covariantly finite subcategory $\mathcal{Y}$, we have notions of {\em $\mathcal{Y}$-coresolution}, {\em $\mathcal{Y}$-coresolution dimension}
$\mathcal{Y}\text{-}{\rm codim}M$ of any object $M\in \mathcal{A}$, and the {\em global $\mathcal{Y}$-coresolution dimension} $\mathcal{Y}\text{-}{\rm codim}\mathcal{A}'$ of any subcategory $\mathcal{A}'$ of $\mathcal{A}$.

Recall that a complex $C$ is {\em right $\mathcal{X}$-acyclic} if it becomes acyclic by applying ${\rm Hom}_{\mathcal{A}}(X, -)$ for any $X\in \mathcal{X}$, and dually, a complex $C$ is {\em left $\mathcal{Y}$-acyclic} if it becomes acyclic by applying ${\rm Hom}_{\mathcal{A}}(-,Y)$ for any $Y\in \mathcal{Y}$; cf. \cite[p. 2721]{Chen10}.
It follows from \cite[Definition 1.1]{Chen10} that a pair $(\mathcal{X}, \mathcal{Y})$ of subcategories in $\mathcal{A}$ is a {\em balanced pair} if $\mathcal{X}$ is contravariantly finite and $\mathcal{Y}$ is covariantly finite; for each object $M\in \mathcal{A}$, there is a left $\mathcal{Y}$-acyclic $\mathcal{X}$-resolution $X^{\bullet}\rightarrow M$, and a right  $\mathcal{X}$-acyclic $\mathcal{Y}$-coresolution $M\rightarrow Y^{\bullet}$.

The balanced pair $(\mathcal{X}, \mathcal{Y})$ is {\em admissible} if each right $\mathcal{X}$-approximation is an epimorphism and each left $\mathcal{Y}$-approximation is a monomorphism.
By \cite[Proposition 2.6]{Chen10}, if there exist two complete and hereditary cotorsion pairs $(\mathcal{X}, \mathcal{Z})$ and $(\mathcal{Z}, \mathcal{Y})$ in $\mathcal{A}$, then $(\mathcal{X}, \mathcal{Y})$ is an admissible balanced pair. In this case, $(\mathcal{X}, \mathcal{Z}, \mathcal{Y})$ is called a {\em complete and hereditary cotorsion triple}. It follows from \cite[Theorem 4.4]{EPZ20} that the existence of complete and hereditary cotorsion triple in $\mathcal{A}$ is equivalent to that $\mathcal{A}$ has enough projective objects and enough injective objects.
\vspace{0.1in}

\noindent{\bf 2.3. Hovey triples and model structures on exact categories.}
We refer to \cite{Qui67,Hov99} for details on model structures and homotopy categories. Let $(\mathcal{A}, \mathcal{E})$ be an exact category equipped with a model structure. An object $M\in \mathcal{A}$ is {\em trivial} if $0\rightarrowtail M$ (equivalently, $M\twoheadrightarrow 0$) is a weak equivalence; it is {\em cofibrant} if $0\rightarrowtail M$ is a cofibration, and it is {\em fibrant} if $M\twoheadrightarrow 0$ is a fibration. The subcategories of trivial, cofibrant and fibrant objects are denoted by $\mathcal{A}_{tri}$, $\mathcal{A}_{c}$ and $\mathcal{A}_{f}$, respectively. Recall that a subcategory $\mathcal{W}$ is {\em thick} if it is closed under direct summands and satisfies the two-out-of-three property for short exact sequences. A triple $(\mathcal{C}, \mathcal{W}, \mathcal{F})$ of subcategories in $(\mathcal{A}, \mathcal{E})$ is a {\em Hovey triple}, if $\mathcal{W}$ is thick and both $(\mathcal{C}, \mathcal{W}\cap \mathcal{F})$ and $(\mathcal{C}\cap \mathcal{W}, \mathcal{F})$ are complete cotorsion pairs. The Hovey triple is {\em hereditary} when these cotorsion pairs are hereditary.

By \cite[Theorem 2.2]{Hov02}, there is a one-to-one correspondence between Hovey triple and model structure for abelian categories; a version for exact categories is from \cite[Theorem 3.3]{Gil11}: a weakly idempotent complete exact category $(\mathcal{A}, \mathcal{E})$ admits a model structure if and only if there is a Hovey triple $(\mathcal{A}_{c}, \mathcal{A}_{tri}, \mathcal{A}_{f})$. In this case, a map is a (trivial) cofibration if and only if it is an admissible monomorphism with a (trivially) cofibrant cokernel, and a map is a (trivial) fibration if and only if it is an admissible epimorphism with a (trivially) fibrant kernel. A map is weak equivalence if and only if it factors as a trivial cofibration followed by a trivial fibration. Throughout this paper, we identify a model structure with its associated Hovey triple.

For a model category $\mathcal{A}$ with model structure $\mathcal{M} = (\mathcal{A}_{c}, \mathcal{A}_{tri}, \mathcal{A}_{f})$, the homotopy category $\mathrm{Ho}(\mathcal{M})$ is the localization of $\mathcal{A}$ with respect to weak equivalences. If $\mathcal{M}$ is hereditary, then the subcategory $\mathcal{A}_{cf} = \mathcal{A}_{c}\cap\mathcal{A}_{f}$ of cofibrant-fibrant objects is a Frobenius category, with $\omega = \mathcal{A}_{c}\cap\mathcal{A}_{tri}\cap\mathcal{A}_{f}$ being the class of projective-injective objects. The stable category $\underline{\mathcal{A}_{cf}} = \mathcal{A}_{cf}/\omega$ is a triangulated category, and moreover, there is a triangle-equivalence $\mathrm{Ho}(\mathcal{M})\simeq \underline{\mathcal{A}_{cf}}$ (see \cite[Theorem 1.2.10]{Hov99}, \cite[Proposition 4.4]{Gil11} or \cite[Proposition 1.1.13]{Bec14}).
\vspace{0.1in}

\noindent{\bf 2.4. Complexes.}
Let $\mathcal{A}$ be an abelian category. A complex $C$ is a sequence $\cdots\rightarrow C_{n+1}\stackrel{d_{n+1}}\rightarrow C_n\stackrel{d_n}\rightarrow C_{n-1}\rightarrow \cdots $ of objects in $\mathcal{A}$ with $d_nd_{n+1} = 0$ for all $n$. An object $M\in A$ is identify with the complex concentrated in degree zero. A chain map $f: C\rightarrow D$ of complexes consists of morphisms $f_n: C_n\rightarrow D_n$ satisfying $f_{n-1}d^C_n = d^D_nf_n$. The category of complexes ${\rm Ch}(\mathcal{A})$ is also an abelian category. For an exact category $(\mathcal{A}, \mathcal{E})$, the category of complexes ${\rm Ch}(\mathcal{A}, \mathcal{E})$ is exact with respect to the class ${\rm Ch}(\mathcal{E})$ of degreewise $\mathcal{E}$-sequences.

A chain map $f: C\rightarrow D$ is {\em null homotopic}, denoted by $f\sim 0$, if there exist maps $s_n: C_n\rightarrow D_{n+1}$ such that $f_n = d^D_{n+1}s_n + s_{n-1}d^C_n$. Two chain maps $f, g: C\rightarrow D$ are {\em chain homotopic} (denoted by $f\sim g$) if $f-g \sim 0$, i.e., $f-g = ds + sd$. The collection $\{s_n\}$ is a {\em chain homotopy} from $f$ to $g$.

Given any object $A\in \mathcal{A}$, let ${\rm S}^nA$ denote the complex with $A$ in degree $n$ and all other entries 0, and let ${\rm D}^nA$ denote the complex with $A$ in degrees $n$ and $n-1$ and all other entries 0, with all maps 0 except $d_n = 1_A$. We refer to \cite[Lemma 3.1]{Gil04} and \cite[Lemma 4.2]{Gil11} for some useful isomorphisms with respect to complexes of the forms ${\rm S}^nA$ and ${\rm D}^nA$.

Given two complexes $C$ and $D$, the {\em Hom-complex} ${\rm Hom}_{\mathcal{A}}(C, D)$ is defined with $n$th component
${\rm Hom}_{\mathcal{A}}(C, D)_{n} = \prod_{k\in \mathbb{Z}}{\rm Hom}_{\mathcal{A}}(X_{k}, Y_{k+n})$ and differential $(\delta_{n}f)_{k} = d_{k+n}^{D}f_{k} - (-1)^{n}f_{k-1}d_{k}^{C}$ for morphisms $f_{k}: C_{k}\rightarrow D_{k+n}$.

Recall that ${\rm Ext}_{{\rm Ch}(\mathcal{A})}^1(C, D)$ is the group of equivalence classes of short exact sequences $0\rightarrow D\rightarrow E\rightarrow C\rightarrow 0$ of complexes. Within this group, ${\rm Ext}_{dw}^1(C, D)$ consists of the degree-wise split sequences, while ${\rm Ext}_{{\rm Ch}(\mathcal{E})}^1(C, D)$ comprises those sequences lying in $\mathcal{E}$ in each degree.

The {\em suspension functor} $\Sigma$ on the category of complexes is defined by $(\Sigma C)_n = C_{n-1}$ and $d_n^{\Sigma C} = - d_{n-1}^C$; $\Sigma^n$ denotes its $n$th power. For any complex $C$, the {\em $n$th homology} ${\rm H}_nC$ is defined to be ${\rm Z}_nC/{\rm B}_nC$, where ${\rm Z}_nC = {\rm Ker}d_n$ is the {\em $n$th cycle}, and ${\rm B}_nC = {\rm Im}d_{n+1}$ is the {\em $n$th boundary}. The following isomorphisms are standard (see, e.g. \cite[Lemma 2.1]{Gil04}).

\begin{lemma}\label{lem:Gil04}
For chain complexes $C$ and $D$, one has
$${\rm Ext}_{dw}^1(C, \Sigma^{-n-1}D) \cong {\rm H}_n{\rm Hom}_\mathcal{A}(C, D) = {\rm Hom}_{{\rm Ch}(\mathcal{A})}(C, \Sigma^{-n}D)/\sim.$$
\end{lemma}

\section{Models for relative derived categories}\label{relative-derived-categories}\label{Models-chain-homotopy}

\noindent Throughout the paper, let $\mathcal{A}$ be a complete and cocomplete abelian category, and let $(\mathcal{X}, \mathcal{Y})$ be an admissible balanced pair in $\mathcal{A}$. We emphasize that both $\mathcal{X}$ and $\mathcal{Y}$ are closed under taking direct summands. We begin with the following proposition.

\begin{proposition}\label{lem:XY-ac}
Let $A$ be a complex in ${\rm Ch}(\mathcal{A})$.
\begin{enumerate}
\item
$A$ is right $\mathcal{X}$-acyclic if and only if it is left $\mathcal{Y}$-acyclic;
\item
If $A$ is right $\mathcal{X}$-acyclic (equivalently, left $\mathcal{Y}$-acyclic), then it is acyclic;
\item
The class of right $\mathcal{X}$-acyclic complexes is closed under direct sums and direct products.
\end{enumerate}

\begin{proof}
(1) The result is a direct consequence of \cite[Proposition 2.2]{Chen10}.

(2) Let $\alpha_n$ be the canonical morphism $A_n\rightarrow {\rm Z}_{n-1}A$. Since $A$ is right $\mathcal{X}$-acyclic, the map ${\rm Hom}_{\mathcal{A}}(X, \alpha_n)$ is surjective. It suffices to prove that the $\alpha_n$ is epic. Denote by $\pi_n$ the right $\mathcal{X}$-approximation of $A_n$, one can check that $\alpha_n\pi_n$ is a right $\mathcal{X}$-approximation of ${\rm Z}_{n-1}A$, thus $\alpha_n$ is epic as $(\mathcal{X}, \mathcal{Y})$ is admissible.

(3) Let $\{A^i\}_{i\in I}$ be a family of right $\mathcal{X}$-acyclic complexes in ${\rm Ch}(\mathcal{A})$. Then they are left $\mathcal{Y}$-acyclic. For any object $Y\in \mathcal{Y}$, one has that ${\rm Hom}_{\mathcal{A}}(A^i, Y)$ is acyclic for any $i\in I$. Hence ${\rm Hom}_{\mathcal{A}}(\bigoplus_{i}A^i, Y) \cong \prod_{i}{\rm Hom}_{\mathcal{A}}(A^i, Y)$ is acyclic, which implies that $\bigoplus_{i}A^i$ is left $\mathcal{Y}$-acyclic. One can check similarly that $\prod_{i}A^i$ is right $\mathcal{X}$-acyclic.
\end{proof}

\end{proposition}

In the following, we denote by $\mathcal{E}$ the class of short exact sequences in $\mathcal{A}$ which are right $\mathcal{X}$-acyclic (equivalently, left $\mathcal{Y}$-acyclic). Then $(\mathcal{A}, \mathcal{E})$ is an exact category and ${\rm Ch}(\mathcal{A}, \mathcal{E})$ is the exact category of complexes, with respect to the short exact sequences of complexes which are right $\mathcal{X}$-acyclic in each degree. Similar to \cite[Definition 3.3]{Gil04}, let $\widetilde{\mathcal{E}}$ be the class of right $\mathcal{X}$-acyclic complexes; the class of complexes $X\in\widetilde{\mathcal{E}}$ with all ${\rm Z}_nX \in \mathcal{X}$ is denoted by $\widetilde{\mathcal{X}}_{\mathcal{E}}$, and $\mathcal{E}\text{-}{\rm dg}\widetilde{\mathcal{X}}$ denotes the class of those complexes $X$ for which $X_n\in \mathcal{X}$ and every map $X\rightarrow E$ is null homotopic whenever $E\in \widetilde{\mathcal{E}}$. Dually, $\widetilde{\mathcal{Y}}_\mathcal{E}$ and $\mathcal{E}\text{-}{\rm dg}\widetilde{\mathcal{Y}}$ are defined.

For any $M\in \mathcal{X}$, it is easy to see that ${\rm S}^nM\in\mathcal{E}\text{-}{\rm dg}\widetilde{\mathcal{X}}$, and ${\rm D}^nM\in \widetilde{\mathcal{X}}_\mathcal{E}$.
Let $0\rightarrow X'\rightarrow X''\rightarrow X\rightarrow 0$ be a short exact sequence in ${\rm Ch}(\mathcal{A}, \mathcal{E})$ with $X\in \mathcal{E}\text{-}{\rm dg}\widetilde{\mathcal{X}}$. Then $X'\in \mathcal{E}\text{-}{\rm dg}\widetilde{\mathcal{X}}$ if and only if $X''\in \mathcal{E}\text{-}{\rm dg}\widetilde{\mathcal{X}}$.

\begin{proposition}\label{prop:(dgX,E)}
$(\mathcal{E}\text{-}{\rm dg}\widetilde{\mathcal{X}}, \widetilde{\mathcal{E}})$ is a cotorsion pair in ${\rm Ch}(\mathcal{A}, \mathcal{E})$.
\end{proposition}

\begin{proof}
Note that for any $X\in\mathcal{E}\text{-}{\rm dg}\widetilde{\mathcal{X}}$ and any $E\in \widetilde{\mathcal{E}}$, we have
$${\rm Ext}_{{\rm Ch}(\mathcal{E})}^1(X, E) = {\rm Ext}_{dw}^1(X, E) \cong
{\rm Hom}_{{\rm Ch}(\mathcal{A})}(X, \Sigma E)/\sim = 0,$$
where $\sim$ is chain homotopy. This implies that $\mathcal{E}\text{-}{\rm dg}\widetilde{\mathcal{X}} \subseteq {^{\perp}\widetilde{\mathcal{E}}}$, and
$(\mathcal{E}\text{-}{\rm dg}\widetilde{\mathcal{X}})^{\perp} \supseteq \widetilde{\mathcal{E}}$.

Let $E\in (\mathcal{E}\text{-}{\rm dg}\widetilde{\mathcal{X}})^{\perp}$. Consider the sequence $0 \rightarrow {\rm Z}_nE\rightarrow E_n\stackrel{\alpha}\rightarrow {\rm Z}_{n-1}E\rightarrow 0$. Let $f: X\rightarrow {\rm Z}_{n-1}E$ be any map for which $X\in \mathcal{X}$. By \cite[Lemma 3.1(3)]{Gil04}, $f$ induces a unique chain map $\widetilde{f}: {\rm S}^{n-1}X\rightarrow E$. Indeed, it is easy to check that $\widetilde{f}_{n-1} = \iota f$ where $\iota: {\rm Z}_{n-1}E\rightarrow E_{n-1}$, and $\widetilde{f}_i =0$ for all $i\neq n-1$. Since ${\rm S}^{n-1}X \in \mathcal{E}\text{-}{\rm dg}\widetilde{\mathcal{X}}$, $\widetilde{f}$ is a null homotopy, and then there is a map $g: X\rightarrow E_n$ such that $f = \alpha g$. This implies that each $0 \rightarrow {\rm Z}_nE\rightarrow E_n\rightarrow {\rm Z}_{n-1}E\rightarrow 0$ is right $\mathcal{X}$-acyclic, and then $E\in \widetilde{\mathcal{E}}$.

Let $C\in {^{\perp}\widetilde{\mathcal{E}}}$. Since $\mathcal{X}$ is contravariantly finite, for any $C_n$ there is a right $\mathcal{X}$-acyclic sequence $0 \rightarrow K\rightarrow X_n \stackrel{p}\rightarrow C_n\rightarrow 0$ for which $X_n\in \mathcal{X}$. Consider the pullback of $p: X_n\rightarrow C_n$ and $d_{n+1}: C_{n+1}\rightarrow C_n$:
$$\xymatrix@C=22pt@R=20pt{ 0 \ar[r] &K \ar[r]^{} \ar@{=}[d] &M \ar[r]^{q\ \ }\ar[d]^{\delta} &C_{n+1}\ar[d]^{d_{n+1}} \ar[r] &0\\
0 \ar[r] &K\ar[r] & X_n\ar[r]^{p} & C_n \ar[r] &0 .}$$
Since $d_{n+1}d_{n+2} = 0$, there exists a map $\eta: C_{n+2}\rightarrow M$ such that $\delta\eta = 0$ and $q\eta = d_{n+2}$. By the universal property of pullback, it follows that $\eta d_{n+3} = 0$.
Hence, we can construct a complex
$$D = \cdots \longrightarrow C_{n+3} \stackrel{d_{n+3}}\longrightarrow C_{n+2} \stackrel{\eta}\longrightarrow M \stackrel{\delta} \longrightarrow X_n \stackrel{d_np} \longrightarrow C_{n-1} \stackrel{d_{n-1}} \longrightarrow C_{n-2}\longrightarrow \cdots $$
and get a short exact sequence $0\rightarrow {\rm D}^{n+1}K\rightarrow D\rightarrow C\rightarrow 0$ in ${\rm Ch}(\mathcal{A}, \mathcal{E})$. Note that ${\rm D}^{n+1}K \in \widetilde{\mathcal{E}}$. It follows from ${\rm Ext}_{{\rm Ch}(\mathcal{E})}^1(C, {\rm D}^{n+1}K) = 0$ that the above sequence of complexes is split, and in particular, the short exact sequence  $0 \rightarrow K\rightarrow X_n \stackrel{p}\rightarrow C_n\rightarrow 0$ is split. This implies that $C_n \in \mathcal{X}$ for any $n$.

For any $E\in \widetilde{\mathcal{E}}$, as ${\rm Hom}_{{\rm Ch}(\mathcal{A})}(C, E)/\sim\cong{\rm Ext}_{{\rm Ch}(\mathcal{E})}^1(C, \Sigma^{-1} E) = 0$, it implies that $C\in \mathcal{E}\text{-}{\rm dg}\widetilde{\mathcal{X}}$. Hence $(\mathcal{E}\text{-}{\rm dg}\widetilde{\mathcal{X}}, \widetilde{\mathcal{E}})$ is a cotorsion pair in ${\rm Ch}(\mathcal{A}, \mathcal{E})$.
\end{proof}

In order to prove the completeness of the above cotorsion pair, we need to make some preparations.

\begin{lemma}\label{admissiblebalancedpairs}
\begin{enumerate}
\item
If a short exact sequence $0\rightarrow A\rightarrow B\rightarrow C\rightarrow 0$ in $\mathcal{A}$ is right $\mathcal{X}$-acyclic, then it remains acyclic after applying ${\rm Hom}_\mathcal{A}(\bigoplus_{i}X^i, -)$ for any family of objects $\{X^{i}\}_{i\in I}$ in $\mathcal{X}$;
\item If an object $Z\in\mathcal{A}$ has the property that ${\rm Hom}_\mathcal{A}(Z, -)$ sends every right $\mathcal{X}$-acyclic complex to an acyclic complex, then $Z\in \mathcal{X}$. Consequently, $\mathcal{X}$ is closed under direct sums;
\item
If $X^{1}\stackrel{f^{1}}\longrightarrow X^{2}\stackrel{f^{2}}\longrightarrow X^{3}\longrightarrow \cdots$ is a direct system of objects in $\mathcal{X}$ with each $f^{i}$ split monic, then the direct limit $\varinjlim_{i\geq 1} X^{i} \in \mathcal{X}$.
\end{enumerate}
\end{lemma}

\begin{proof}
One may verify (1) directly. Since the right $\mathcal{X}$-approximation of $Z$ is split epic by definition, it follows that $Z \in \mathcal{X}$. Consequently, $\mathcal{X}$ is closed under direct sums by (1).

(3) Denote by $Y^i$ the cokernel of $f^i$. Then for all $i\ge 1$ we have $X^{i+1}\cong X^i\bigoplus Y^i$ with $Y^i \in \mathcal{X}$. One can check that the direct limit of the direct system is exactly $X^1\bigoplus Y^1\bigoplus Y^2\bigoplus \cdots,$ which lies in $\mathcal{X}$ by (2).
\end{proof}

A complex $(C,d)$ is said to be {\em right $\mathcal{X}$-acyclic} {\em at degree $n$}, if for any $X\in \mathcal{X}$, the complex ${\rm Hom}_{\mathcal{A}}(X, C)$ is acyclic at ${\rm Hom}_{\mathcal{A}}(X, C_n)$. One can check directly that $(C,d)$ is right $\mathcal{X}$-acyclic at degree $n$ if and only if the sequence $0\rightarrow {\rm Hom}_{\mathcal{A}}(X,{\rm Z}_{n+1}C) \rightarrow {\rm Hom}_{\mathcal{A}}(X, C_{n+1})\rightarrow {\rm Hom}_{\mathcal{A}}(X,{\rm Z}_nC)\rightarrow 0$ is exact for any $X\in \mathcal{X}$, that is, $0\rightarrow {\rm Z}_{n+1}C \rightarrow C_{n+1}\rightarrow {\rm Z}_nC\rightarrow 0$ is right $\mathcal{X}$-acyclic.

\begin{lemma}\label{lem:directsystem}
Let $A^{1}\stackrel{f^{1}}\longrightarrow A^{2}\stackrel{f^{2}}\longrightarrow A^{3}\longrightarrow \cdots$ be a direct system in ${\rm Ch}(\mathcal{A})$ with $A =\varinjlim_{i\geq 1} A^{i}$ such that each $f^{i}$ is monic and degree-wise split.
\begin{enumerate}
\item
For $k\in \mathbb{Z}$, if $A^i$ is right $\mathcal{X}$-acyclic at degrees $k,k-1,k-2$ and $k-3$ for any $i\geq1$, then $A$ is right $\mathcal{X}$-acyclic at degree $k$;
\item
If $A^{i}$ lies in $\mathcal{E}\text{-}{\rm dg}\widetilde{\mathcal{X}}$ for any $i\geq1$, then $A\in \mathcal{E}\text{-}{\rm dg}\widetilde{\mathcal{X}}$.
\end{enumerate}
\end{lemma}

\begin{proof}
(1) Since $A^i$ is right $\mathcal{X}$-acyclic at degree $k$, we obtain a short exact sequence $0\rightarrow {\rm Z}_{k+1}A^i \rightarrow A^i_{k+1}\rightarrow {\rm Z}_kA^i\rightarrow 0$ which is right $\mathcal{X}$-acyclic. By Proposition \ref{lem:XY-ac}(3), one has that $0\rightarrow \bigoplus_{i}{\rm Z}_{k+1}A^i \rightarrow \bigoplus_{i}A^i_{k+1}\rightarrow \bigoplus_{i}{\rm Z}_kA^i\rightarrow 0$ is also right $\mathcal{X}$-acyclic.

Considering degrees $k-1$ and $k-2$ in a similar method, one has the following right $\mathcal{X}$-acyclic complexes
$$0\rightarrow \bigoplus_{i}{\rm Z}_{k}A^i \rightarrow \bigoplus_{i}A^i_{k}\rightarrow \bigoplus_{i}{\rm Z}_{k-1}A^i\rightarrow 0,$$
$$0\rightarrow \bigoplus_{i}{\rm Z}_{k-1}A^i \rightarrow \bigoplus_{i}A^i_{k-1}\rightarrow \bigoplus_{i}{\rm Z}_{k-2}A^i\rightarrow 0.$$
Applying the functor ${\rm Hom}_{\mathcal{A}}(X, -)$ and pasting the sequences together, we conclude that the complex $\bigoplus_{i}A^i$ is right $\mathcal{X}$-acyclic at degree $k$. Note that $A^i$ is right $\mathcal{X}$-acyclic at degrees $k-1,k-2$ and $k-3$  for any $i\geq1$. A similar argument yields that $\bigoplus_{i}A^i$ is also right $\mathcal{X}$-acyclic at degree $k-1$.

Denote by $\nu: \bigoplus_{i}A^i\rightarrow \bigoplus_{i}A^i$ the morphism induced by $f^i$. Then $A = {\rm Coker}(1-\nu)$ and the short exact sequence of complexes
\begin{equation}
0\longrightarrow \bigoplus_{i}A^i \stackrel{1-\nu}\longrightarrow \bigoplus_{i}A^i\longrightarrow A\longrightarrow 0 \tag{$*$}
\end{equation}
is degree-wise split by \cite[Lemma 64]{Mur06}. For any $X\in \mathcal{X}$, by applying ${\rm Hom}_{\mathcal{A}}(X, -)$ we get a degree-wise split short exact sequence of complexes of abelian groups
$$0\longrightarrow {\rm Hom}_{\mathcal{A}}(X, \bigoplus_{i}A^i) \longrightarrow {\rm Hom}_{\mathcal{A}}(X, \bigoplus_{i}A^i)\longrightarrow {\rm Hom}_{\mathcal{A}}(X, A)\longrightarrow 0,$$
where ${\rm Hom}_{\mathcal{A}}(X, \bigoplus_{i}A^i)$ is acyclic at degrees $k$ and $k-1$ by argument above. Thus ${\rm Hom}_{\mathcal{A}}(X, A)$ is acyclic at degree $k$ by the long exact sequence in homology, as desired.

(2) For $E\in \widetilde{\mathcal{E}}$, the sequence $(*)$ induces a short exact sequence of Hom-complexes
$$0 \rightarrow {\rm Hom}_{\mathcal{A}}( A, E)\rightarrow {\rm Hom}_{\mathcal{A}}(\bigoplus_{i} A^{i},E)\rightarrow {\rm Hom}_{\mathcal{A}}(\bigoplus_{i} A^{i},E) \rightarrow 0.$$
Since $\prod_{i}{\rm Hom}_{\mathcal{A}}( A^{i},E) \cong {\rm Hom}_{\mathcal{A}}(\bigoplus_{i} A^{i},E)$ is acyclic, so is ${\rm Hom}_{\mathcal{A}}( A, E)$. This implies that every chain map $A\rightarrow E$ is null homotopic by Lemma \ref{lem:Gil04}. The rest of the assertions come from Lemma \ref{admissiblebalancedpairs}(3).
\end{proof}

Recall that a \emph{left brutal truncation} (resp. \emph{right truncation}) at degree $i$ of a complex $(C,d)$ is a complex of the form $0\rightarrow C_i \rightarrow C_{i-1} \rightarrow C_{i-2} \rightarrow \cdots$ (resp. $\cdots \rightarrow C_{i+2}\rightarrow C_{i+1}\rightarrow {\rm Ker}d_i\rightarrow 0$).
For convenience, in what follows we denote a direct system $K^{1}\stackrel{f^{1}}\longrightarrow K^{2}\stackrel{f^{2}}\longrightarrow K^{3}\longrightarrow \cdots$ simply by $(K^n, f^n)$.

\begin{lemma}\label{lem:directlimit}
Let $(A^n, f^n)$, $(B^n, g^n)$ and $(C^n, h^n)$ be three direct systems of complexes in ${\rm Ch}(\mathcal{A})$ such that for each $n\ge 1$, $0\rightarrow A^n\rightarrow B^n\rightarrow C^n\rightarrow 0$ is a degree-wise left $\mathcal{Y}$-acyclic short exact sequence of complexes.

Assume further that for every integer $i$, there exists $m\gg 0$ such that for each $n\geq m$, $h^n_i:C^n_i\rightarrow C^{n+1}_i$ remains epic by applying ${\rm Hom}_{\mathcal{A}}(-,\mathcal{Y})$.
Then there is a short exact sequence $$0\rightarrow A\rightarrow B\rightarrow C\rightarrow 0$$ in ${\rm Ch}(\mathcal{A}, \mathcal{E})$ induced by the direct limits $A=\varinjlim_{n\geq 1}A^n$, $B=\varinjlim_{n\geq 1}B^n$ and $C=\varinjlim_{n\geq 1}C^n$.
In particular, if for each $n\geq 1$, $C^n$ is the left brutal truncation at degree $n$ or the right truncation at degree $-n$ of a complex $C$, then $0\longrightarrow A \longrightarrow B\longrightarrow C\longrightarrow 0$ is exact  in ${\rm Ch}(\mathcal{A}, \mathcal{E})$.
\end{lemma}

\begin{proof}
It is easy to check that for any integer $s> 1$, $\varinjlim_{j\geq s} A^{j} = A$, $\varinjlim_{j\geq s} B^{j} = B$ and $\varinjlim_{j\geq s} C^{j} = C$.

Given an integer $i$, let $m$ be as assumed. For any $Y\in \mathcal{Y}$, by applying ${\rm Hom}_{\mathcal{A}}(-, Y)$ to  the $\geq m$ part of $i$-th degree of the short exact sequence of complexes, we get the following diagram
$$\xymatrix@C=22pt@R=20pt{ & 0\ar[d] & 0\ar[d] & 0\ar[d] \\
\cdots \ar[r] & {\rm Hom}_{\mathcal{A}}(C^{m+2}_i, Y) \ar@{->>}[r] \ar[d]  & {\rm Hom}_{\mathcal{A}}(C^{m+1}_i, Y) \ar[d]\ar@{->>}[r] &  {\rm Hom}_{\mathcal{A}}(C^m_i, Y) \ar[d]  \\
\cdots \ar[r] & {\rm Hom}_{\mathcal{A}}(B^{m+2}_i, Y) \ar[r] \ar[d]  & {\rm Hom}_{\mathcal{A}}(B^{m+1}_i, Y) \ar[d]\ar[r] &  {\rm Hom}_{\mathcal{A}}(B^m_i, Y) \ar[d]  \\
\cdots \ar[r] & {\rm Hom}_{\mathcal{A}}(A^{m+2}_i, Y) \ar[r]\ar[d] & {\rm Hom}_{\mathcal{A}}(A^{m+1}_i, Y) \ar[r]\ar[d] & {\rm Hom}_{\mathcal{A}}(A^m_i, Y)\ar[d] \\
 & 0 & 0 & \;0,
  }$$
where all maps in the first row are epic. Since $0\longrightarrow A^j \longrightarrow B^j\longrightarrow C^j\longrightarrow 0$ is left $\mathcal{Y}$-acyclic in each degree, all the columns are exact. Then we have the following commutative diagram
$$\xymatrix@C=8pt@R=8pt{
0\ar[r] &\varprojlim_{n\geq m}{\rm Hom}_{\mathcal{A}}(C^n_i, Y) \ar[r]\ar[d]^{\cong} &\varprojlim_{n\geq m}{\rm Hom}_{\mathcal{A}}(B^n_i, Y) \ar[r]\ar[d]^{\cong} &\varprojlim_{n\geq m}{\rm Hom}_{\mathcal{A}}(A^n_i, Y) \ar[d]^{\cong}\ar[r] &0\\
0\ar[r]&{\rm Hom}_{\mathcal{A}}(\varinjlim_{n\geq m}C^n_i, Y)\ar[r]&{\rm Hom}_{\mathcal{A}}(\varinjlim_{n\geq m}B^n_i, Y) \ar[r] &{\rm Hom}_{\mathcal{A}}(\varinjlim_{n\geq m}A^n_i, Y)
\ar[r] &0. }$$
Since the first row in the above diagram is exact by \cite[Theorem 1.6.13]{EJ00}, so is the second row, and therefore the $i$-th degree of $0\longrightarrow A \longrightarrow B\longrightarrow C\longrightarrow 0$ is left $\mathcal{Y}$-acyclic.

In the special case of $(C^n, f^n)$ being the left brutal truncations or the right truncation of $C$, for every integer $i$, there exists $m\gg 0$ such that for all $n\geq m$, $h^n_i$ becomes identity, which satisfies the assumption. This completes the proof.
\end{proof}

\begin{lemma}\label{lem:l-E-ap}
Let $C$ be any complex. Then there exists a short exact sequence $0\rightarrow C\rightarrow E\rightarrow X\rightarrow 0$ in ${\rm Ch}(\mathcal{A}, \mathcal{E})$ with $E \in \widetilde{\mathcal{E}}$ and $X\in \mathcal{E}\text{-}{\rm dg}\widetilde{\mathcal{X}}$.
\end{lemma}

\begin{proof}
Let $n$ be an integer. For the complex $(C, d)$, denote by $\iota_n:{\rm Z}_nC\rightarrow C_n$ the canonical morphism. Choose a right $\mathcal{X}$-approximation $\pi_{n+1}: X_{n+1}\rightarrow {\rm Z}_nC$. Then we have a complex $$(\overline{C}, \delta) =  \cdots \rightarrow C_{n+3} \rightarrow C_{n+2}\rightarrow C_{n+1}\oplus X_{n+1}\rightarrow C_n \rightarrow C_{n-1}\rightarrow\cdots,$$ where $\delta_{n+2} = \left(\begin{smallmatrix}d_{n+2}\\0 \end{smallmatrix}\right)$, $\delta_{n+1} = (d_{n+1}, \iota_n\pi_{n+1})$, and $\delta_i = d_i$ for any $i\neq n+1, n+2$. Note that $\overline{C}$ is acyclic at degree $n$, that is,
${\rm B}_n\overline{C} = {\rm Z}_nC = {\rm Z}_n\overline{C}$. Since $\pi_{n+1}: X_{n+1}\rightarrow {\rm Z}_nC$ is a right $\mathcal{X}$-approximation, the sequence $C_{n+1}\oplus X_{n+1} \rightarrow {\rm Z}_nC \rightarrow 0$ is right $\mathcal{X}$-acyclic. Then, $\overline{C}$ is right $\mathcal{X}$-acyclic at degree $n$. Moreover, there is a chain map $f: C\rightarrow \overline{C}$ with $f_{n+1} = \left(\begin{smallmatrix} 1 \\0 \end{smallmatrix}\right)$ and $f_i = \text{Id}$ for all $i\neq n+1$. Then $f$ is a degree-wise split monomorphism with ${\rm Coker}f = {\rm S}^{n+1}X_{n+1} \in \mathcal{E}\text{-}{\rm dg}\widetilde{\mathcal{X}}$.

Similar to the above argument, one has a complex
$$\overline{\overline{C}} =  \cdots \rightarrow \overline{C}_{n+3} \rightarrow \overline{C}_{n+2}\oplus X_{n+2}\rightarrow \overline{C}_{n+1}\rightarrow \overline{C}_n \rightarrow \overline{C}_{n-1}\rightarrow\cdots $$
such that it is right $\mathcal{X}$-acyclic at both degrees $n+1$ and $n$, there exists a degree-wise split and monic chain map $\overline{f}:\overline{C} \rightarrow \overline{\overline{C}}$ with ${\rm Coker}\overline{f} = {\rm S}^{n+2}X_{n+2}  \in \mathcal{E}\text{-}{\rm dg}\widetilde{\mathcal{X}}$. Then we have the following commutative diagram
$$\xymatrix@C=22pt@R=20pt{ & & 0\ar[d] &0\ar[d] \\
0\ar[r] &C \ar[r]^f\ar@{=}[d] &\overline{C} \ar[r]\ar[d]^{\overline{f}} &{\rm Coker}f \ar[r]\ar[d] & 0\\
0\ar[r]&C\ar[r] &\overline{\overline{C}} \ar[r]\ar[d] &L \ar[r]\ar[d] & 0\\
& & {\rm Coker}\overline{f} \ar@{=}[r] \ar[d] &{\rm Coker}\overline{f} \ar[d]\\ & & 0 & \;0,}$$
where the first row, the second row, and the second column are exact in ${\rm Ch}(\mathcal{A}, \mathcal{E})$. Then the third column is also exact in ${\rm Ch}(\mathcal{A}, \mathcal{E})$, which implies that $L\in \mathcal{E}\text{-}{\rm dg}\widetilde{\mathcal{X}}$. Thus $\overline{f}f: C \rightarrow \overline{\overline{C}}$ is a degree-wise split and monic chain map with cokernel in $\mathcal{E}\text{-}{\rm dg}\widetilde{\mathcal{X}}$ and $\overline{\overline{C}}$ right $\mathcal{X}$-acyclic at both degrees $n+1$ and $n$. Repeating the process again we get a degree-wise split and monic chain map $C \rightarrow C'$ with cokernel in $\mathcal{E}\text{-}{\rm dg}\widetilde{\mathcal{X}}$ and $C'$ right $\mathcal{X}$-acyclic at degrees $n+3, n+2, n+1$ and $n$.

Let $(m_1, m_2, m_3, \cdots) = (0, 1, -1, 2, -2, \cdots)$ and $(n_1, n_2, n_3, n_4, n_5, n_6, n_7, \cdots) = (m_1, m_2, m_1, m_2, m_3, m_1, m_2, m_3, m_4, \cdots)$, which means that every integer occurs infinitely often in the original sequence $n_1, n_2, n_3, \cdots$. We perform the above construction to obtain a series of complexes and chain maps $C = C^0$, $C^0\rightarrow C^1$, $C^1\rightarrow C^2$, $C^2\rightarrow C^3, \cdots$, where $C^i$ is right $\mathcal{X}$-acyclic at degrees $n_i, n_i-1, n_i-2$ and $n_i-3$, each chain map is monic and degree-wise split, and $D^i = {\rm Coker}(C^{i-1}\rightarrow C^i)$ is in $\mathcal{E}\text{-}{\rm dg}\widetilde{\mathcal{X}}$ for $i\geq 1$. We claim that for any $i> j$, the cokernel of the composition ${\rm Coker}(C^j\rightarrow C^i)$ lies in $\mathcal{E}\text{-}{\rm dg}\widetilde{\mathcal{X}}$.

Given $j$, set $L^i = {\rm Coker}(C^j\rightarrow C^i)$ for any $i> j$. We have the following commutative diagram
$$\xymatrix@C=22pt@R=20pt{ & & 0\ar[d] &0\ar[d] \\ 0\ar[r] &C^j \ar[r]\ar@{=}[d] &C^i \ar[r]\ar[d]^{f^i} &L^i \ar[r]\ar[d]^{g^i} & 0\\
0\ar[r]&C^j\ar[r] &C^{i+1} \ar[r]\ar[d] &L^{i+1} \ar[r]\ar[d] & 0\\ & & D^{i+1} \ar@{=}[r] \ar[d] &D^{i+1} \ar[d]\\ & & 0 & \;0.}$$
From the right column, we begin with $L^{j+1}\in \mathcal{E}\text{-}{\rm dg}\widetilde{\mathcal{X}}$ to infer inductively that $L^i \in \mathcal{E}\text{-}{\rm dg}\widetilde{\mathcal{X}}$ for all $i> j$. Moreover, it is easy to check that each $g^i: L^i\rightarrow L^{i+1}$ is a degree-wise split monomorphism with cokernel in $\mathcal{E}\text{-}{\rm dg}\widetilde{\mathcal{X}}$.

By setting $j = 0$ we get the diagram
$$\xymatrix@C=22pt@R=20pt{ 0\ar[r] &C \ar[r]\ar@{=}[d] &C^i \ar[r]\ar[d]^{f^i} &L^i \ar[r]\ar[d]^{g^i} & 0\\
0\ar[r]&C\ar[r] &C^{i+1} \ar[r] &L^{i+1} \ar[r] & 0.}$$
Let $E = \varinjlim_{i\geq 0} C^i$ and $X = \varinjlim_{i\geq 0} L^i$. Since each $f^i$ and $g^i$ are degree-wise split monic, by Lemma \ref{lem:directlimit} we have a short exact sequence $0\rightarrow C\rightarrow E\rightarrow X\rightarrow 0$ in ${\rm Ch}(\mathcal{A}, \mathcal{E})$. By Lemma \ref{lem:directsystem}(2), we have $X\in \mathcal{E}\text{-}{\rm dg}\widetilde{\mathcal{X}}$.

Moreover, for any $k$, there exists infinitely many natural numbers $s_1, s_2, s_3, \cdots$ such that $C^{s_1}, C^{s_2}, C^{s_3}, \cdots$ are right $\mathcal{X}$-acyclic at degrees $k, k-1, k-2$ and $k-3$. Consider the direct system $C^{s_1}\stackrel{h^{1}}\longrightarrow C^{s_2}\stackrel{h^{2}}\longrightarrow \cdots$, where $h^{i} = f^{s_{i+1}-1}\cdots f^{s_i+1} f^{s_i}$. By Lemma \ref{lem:directsystem}(1), we obtain that $E = \varinjlim_{i\geq 0} C^i = \varinjlim_{i\geq 1} C^{s_i}$ is right $\mathcal{X}$-acyclic at degree $k$. Since $k$ is arbitrary, we have $E\in \widetilde{\mathcal{E}}$, as required.
\end{proof}

Inspired by \cite[Lemma 3.4]{YL11}, we have the following:

\begin{lemma}\label{lem:r-X-ap}
Let $E\in \widetilde{\mathcal{E}}$ be any right $\mathcal{X}$-acyclic complex. Then there exists a short exact sequence $0\rightarrow K\rightarrow X\rightarrow E\rightarrow 0$ in ${\rm Ch}(\mathcal{A}, \mathcal{E})$, for which $X \in \widetilde{\mathcal{X}}_{\mathcal{E}}$ and $K\in \widetilde{\mathcal{E}}$.
\end{lemma}

\begin{proof}
For $n\in \mathbb{Z}$, let $f'_n: X'_n\rightarrow {\rm Z}_nE$ be a right $\mathcal{X}$-approximation of ${\rm Z}_nE$. Since the short exact sequence $0\rightarrow {\rm Z}_nE\rightarrow E_n\rightarrow {\rm Z}_{n-1}E\rightarrow 0$ is right $\mathcal{X}$-acyclic, by \cite[Lemma 8.2.1]{EJ00} we have the following commutative diagram
$$\xymatrix@C=22pt@R=20pt{ 0 \ar[r] &X'_n \ar[r]^{} \ar[d]^{f'_n} &X'_n\oplus X'_{n-1} \ar[r]^{}\ar[d]^{f_n} &X'_{n-1}\ar[d]^{f'_{n-1}} \ar[r] &0\\
0 \ar[r] &{\rm Z}_nE\ar[r] & E_n\ar[r]^{} & {\rm Z}_{n-1}E \ar[r] &0.}$$
Since both $f'_n$ and $f'_{n-1}$ are admissible epic, so is $f_n$.

Let $X_n = X'_n\oplus X'_{n-1}$. By splicing the sequence $0\rightarrow X'_n\rightarrow X_n\rightarrow X'_{n-1}\rightarrow 0$, we get a complex $(X, d)$, where $d = \left(\begin{smallmatrix}0 & 1\\0 & 0\end{smallmatrix}\right)$.
In fact, $X = \bigoplus {\rm D}^{n+1}X'_n$, and then $X \in \widetilde{\mathcal{X}}_\mathcal{E}$.  Moreover, we get a short exact sequence of complexes $0\rightarrow K\rightarrow X\rightarrow E\rightarrow 0$ in ${\rm Ch}(\mathcal{A}, \mathcal{E})$. We infer $K\in \widetilde{\mathcal{E}}$ since both $X$ and $E$ are in $\widetilde{\mathcal{E}}$.
\end{proof}

For any complex $C$ of $R$-modules, Spaltenstein established a quasi-isomorphism $P\rightarrow C$ for which $P$ is a $K$-projective complex; see \cite[Theorem C]{Spa88}. In particular, we may reobtain Spaltenstein's quasi-isomorphism by a different method as follows if the balanced pair is specified as $(\mathcal{P}, \mathcal{I})$; compare to \cite[Theorem 2.2.1]{Gar99}.

\begin{lemma}\label{lem:dgX-ap}
For any complex $C$, there is a short exact sequence $0\rightarrow E'\rightarrow X'\rightarrow C\rightarrow 0$ in ${\rm Ch}(\mathcal{A}, \mathcal{E})$, where $X'\in \mathcal{E}\text{-}{\rm dg}\widetilde{\mathcal{X}}$ and $E' \in \widetilde{\mathcal{E}}$.
\end{lemma}

\begin{proof}
It follows from Lemma \ref{lem:l-E-ap}  that there is a short exact sequence of complexes $0\rightarrow C\rightarrow E\rightarrow X\rightarrow 0$ in ${\rm Ch}(\mathcal{A}, \mathcal{E})$, where $E \in \widetilde{\mathcal{E}}$ and $X\in \mathcal{E}\text{-}{\rm dg}\widetilde{\mathcal{X}}$. For $E$, we obtain from Lemma \ref{lem:r-X-ap} a short exact sequence $0\rightarrow E'\rightarrow X''\rightarrow E\rightarrow 0$ in
${\rm Ch}(\mathcal{A}, \mathcal{E})$, where $X'' \in \widetilde{\mathcal{X}}_{\mathcal{E}}$ and $E'\in \widetilde{\mathcal{E}}$.

Consider the following pullback of $C\rightarrow E$ and $X''\rightarrow E$:
$$\xymatrix@C=22pt@R=20pt{ & 0\ar[d] & 0\ar[d] \\
 & E' \ar@{=}[r]^{} \ar[d]  &E' \ar[d]\\
0 \ar[r] & X' \ar@{-->}[d] \ar@{-->}[r] & X'' \ar[r] \ar[d] & X \ar@{=}[d]  \ar[r] &0 \\
0 \ar[r] & C\ar[r] \ar[d] & E \ar[r]^{} \ar[d] & X \ar[r] & 0\\
  & 0 & \;0.}$$
Since ${\rm Ch}(\mathcal{A}, \mathcal{E})$ is an exact category, we infer that both  $0\rightarrow E'\rightarrow X'\rightarrow C\rightarrow 0$ and
$0\rightarrow X'\rightarrow X''\rightarrow X\rightarrow 0$ are short exact sequences in ${\rm Ch}(\mathcal{A}, \mathcal{E})$. Moreover, since $X$ and $X''$ are in $\mathcal{E}\text{-}{\rm dg}\widetilde{\mathcal{X}}$, it follows that $X'\in \mathcal{E}\text{-}{\rm dg}\widetilde{\mathcal{X}}$. This completes the proof.
\end{proof}

Combining Proposition \ref{prop:(dgX,E)} with Lemmas \ref{lem:l-E-ap} and \ref{lem:dgX-ap}, we obtain the following:

\begin{proposition}\label{thm:(dgX,E)}
$(\mathcal{E}\text{-}{\rm dg}\widetilde{\mathcal{X}}, \widetilde{\mathcal{E}})$ is a complete cotorsion pair in ${\rm Ch}(\mathcal{A}, \mathcal{E})$.
\end{proposition}

We need to prove the following facts.

\begin{lemma}\label{lem:Xe}
\begin{enumerate}
\item Every complex in $\widetilde{\mathcal{X}}_\mathcal{E}$ is of the form $\bigoplus {\rm D}^nX'_n$ for $X'_n\in \mathcal{X}$. Moreover, $\widetilde{\mathcal{X}}_\mathcal{E} = \mathcal{E}\text{-}{\rm dg}\widetilde{\mathcal{X}} \cap \widetilde{\mathcal{E}}$.
\item Let $X$ be a complex with each item in $\mathcal{X}$ and $C$ a complex. For any chain map $f: X\rightarrow C$, $f$ is null homotopic if and only if it can be factored through a complex in $\widetilde{\mathcal{X}}_\mathcal{E}$.
\end{enumerate}
\end{lemma}

\begin{proof}
(1) For any complex $X\in \widetilde{\mathcal{X}}_\mathcal{E}$, there are short exact sequences $0\rightarrow {\rm Z}_nX\rightarrow X_n\rightarrow {\rm Z}_{n-1}X\rightarrow 0$ in $\mathcal{E}$ with ${\rm Z}_nX \in \mathcal{X}$. These sequences are split, and then $X_n = {\rm Z}_nX \oplus {\rm Z}_{n-1}X$. Hence $X = \bigoplus {\rm D}^nX'_n$, where $X'_n = {\rm Z}_{n-1}X$.

It is clear that $\widetilde{\mathcal{X}}_\mathcal{E} \subseteq \mathcal{E}\text{-}{\rm dg}\widetilde{\mathcal{X}} \cap \widetilde{\mathcal{E}}$. For any complex $X\in \mathcal{E}\text{-}{\rm dg}\widetilde{\mathcal{X}} \cap \widetilde{\mathcal{E}}$, we infer from the definition of $\mathcal{E}\text{-}{\rm dg}\widetilde{\mathcal{X}}$ that ${\rm Id}_X\sim 0$, i.e., $X$ is a contractible complex.  Then, $X$ is of the form $\bigoplus {\rm D}^nX'_n$ for some $X'_n\in \mathcal{X}$, and hence $X\in \widetilde{\mathcal{X}}_\mathcal{E}$ holds by the above argument.

(2) The ``if'' part is clear since any complex in $\widetilde{\mathcal{X}}_\mathcal{E}$ is contractible. For the ``only if'', we observe that the cone of ${\rm Id}_X$ is the complex
$$X': =\bigoplus {\rm D}^{n+1}X_n = \cdots \rightarrow X_{n+1}\oplus X_n\rightarrow X_n\oplus X_{n-1}\rightarrow X_{n-1}\oplus X_{n-2}\rightarrow \cdots,$$
which lies in $\widetilde{\mathcal{X}}_\mathcal{E}$. Since $f$ is null homotopic, it factors through $X'$ in ${\bf K}(\mathcal{A})$, as required.
\end{proof}


\begin{proposition}\label{prop:(X,Ch)}
$(\widetilde{\mathcal{X}}_\mathcal{E}, {\rm Ch}(\mathcal{A}))$ is a complete cotorsion pair in ${\rm Ch}(\mathcal{A}, \mathcal{E})$.
\end{proposition}

\begin{proof}
It is clear that $(\widetilde{\mathcal{X}}_\mathcal{E})^\perp = {\rm Ch}(\mathcal{A})$ and  $\widetilde{\mathcal{X}}_\mathcal{E} \subseteq {^\perp{\rm Ch}(\mathcal{A})}$. Let $C\in {^\perp{\rm Ch}(\mathcal{A})}$. By Lemma \ref{lem:dgX-ap} we have a short exact sequence $0\rightarrow E\rightarrow X\rightarrow C\rightarrow 0$ in ${\rm Ch}(\mathcal{A}, \mathcal{E})$ for which $X\in \mathcal{E}\text{-}{\rm dg}\widetilde{\mathcal{X}}$ and $E \in \widetilde{\mathcal{E}}$. Note that the sequence is split, and then as a direct summand of $X$, $C\in \mathcal{E}\text{-}{\rm dg}\widetilde{\mathcal{X}}$. Since ${\rm Id}_C$ is null homotopic, by Lemma \ref{lem:Xe} it factors through a complex in $\widetilde{\mathcal{X}}_\mathcal{E}$, thus $C\in \widetilde{\mathcal{X}}_\mathcal{E}$. Hence, $\widetilde{\mathcal{X}}_\mathcal{E} \supseteq {^\perp{\rm Ch}(\mathcal{A})}$, and finally,  $(\widetilde{\mathcal{X}}_\mathcal{E}, {\rm Ch}(\mathcal{A}))$ is a cotorsion pair in ${\rm Ch}(\mathcal{A}, \mathcal{E})$.

Next, it suffices to prove that for any complex $C$,  there is a short exact sequence $0\rightarrow K\rightarrow X\rightarrow C\rightarrow 0$ in ${\rm Ch}(\mathcal{A}, \mathcal{E})$ for which $X\in \widetilde{\mathcal{X}}_{\mathcal{E}}$. For each term $C_n$ of $C$, there is a right $\mathcal{X}$-approximation $f_n: X_n\rightarrow C_n$, then we have the chain map
$$\xymatrix@C=20pt@R=20pt{ X =\ar[d]_{g} &\cdots \ar[r] & X_{n+1}\oplus X_{n}\ar[r]^{e_n}\ar[d]^{g_n} &X_{n} \oplus X_{n-1}\ar[r]^{e_{n-1}\ \ }\ar[d]^{g_{n-1}} &X_{n-1} \oplus X_{n-2}\ar[r]\ar[d]^{g_{n-2}} &\cdots \\
C = &\cdots \ar[r]&C_n \ar[r]^{d_n}\ar[r] &C_{n-1} \ar[r]^{d_{n-1}}\ar[r] &C_{n-2} \ar[r]\ar[r] & \cdots  }$$
with $e_n = \left(\begin{smallmatrix}0  & 1 \\ 0 &0 \end{smallmatrix}\right)$ and $g_n = \left(d_{n+1}f_{n+1}, f_n \right)$. Since each $f_n$ is admissible epic, so is $g_n$. Set $K$ to be ${\rm Ker}g$, we obtain the desired sequence.
\end{proof}

Recall that a chain map $f: C\rightarrow D$ is said to be an {\em $\mathcal{X}$-quasi-isomorphism} if the resulting chain map ${\rm Hom}_\mathcal{A}(X, f)$ is a quasi-isomorphism for any $X\in \mathcal{X}$. Let ${\bf K}(\mathcal{A})$ be the chain homotopy category of the abelian category $\mathcal{A}$. Denote by $S_\mathcal{X}$ the class of all right $\mathcal{X}$-quasi-isomorphisms in ${\bf K}(\mathcal{A})$. Note that $S_\mathcal{X}$ is a saturated multiplicative system corresponding to the subcategory $\widetilde{\mathcal{E}}$ of ${\bf K}(\mathcal{A})$ in the sense that a chain map $f: C\rightarrow D$ is an $\mathcal{X}$-quasi-isomorphism if and only if its mapping cone ${\rm Con}(f)$ is right $\mathcal{X}$-acyclic (i.e., ${\rm Con}(f)\in \widetilde{\mathcal{E}}$).

The {\em relative derived category}  ${\bf D}_\mathcal{X}(\mathcal{A})$ of $\mathcal{A}$ with respect to $\mathcal{X}$ is defined to be the Verdier quotient of ${\bf K}(\mathcal{A})$ modulo the subcategory $\widetilde{\mathcal{E}}$; see \cite[Definition 3.1]{Chen10}. In symbols, $${\bf D}_\mathcal{X}(\mathcal{A}):= {\bf K}(\mathcal{A})/ \widetilde{\mathcal{E}} = S_\mathcal{X}^{-1}{\bf K}(\mathcal{A}).$$
Note that the relative derived category ${\bf D}_\mathcal{X}(\mathcal{A})$ coincides with Neeman's derived category of the exact category $(\mathcal{A}, \mathcal{E})$ in \cite[Construction 1.5]{Nee90}.

Now, we are in a position to show the model for the relative derived category ${\bf D}_\mathcal{X}(\mathcal{A})$. Let ${\bf K}_{\mathcal{E}\text{-}{\rm dg}}(\mathcal{X})$ be the subcategory of ${\bf K}(\mathcal{A})$ whose objects are complexes in $\mathcal{E}\text{-}{\rm dg}\widetilde{\mathcal{X}}$. By Remark \ref{rem:1}, the following result generalizes \cite[Proposition 3.5]{Chen10}.

\begin{theorem}\label{thm:M4RDC}
There is a model structure $\mathcal{M}_{dg\mathcal{X}} = (\mathcal{E}\text{-}{\rm dg}\widetilde{\mathcal{X}}, \widetilde{\mathcal{E}}, {\rm Ch}(\mathcal{A}))$ on the exact category ${\rm Ch}(\mathcal{A}, \mathcal{E})$. Moreover, we have triangle-equivalences
$${\rm Ho}(\mathcal{M}_{dg\mathcal{X}}) \simeq {\bf K}_{\mathcal{E}\text{-}{\rm dg}}(\mathcal{X}) \simeq {\bf D}_\mathcal{X}(\mathcal{A}).$$
\end{theorem}

\begin{proof}
It is clear that $\widetilde{\mathcal{E}}$ is a thick subcategory. Combining Propositions \ref{thm:(dgX,E)} and \ref{prop:(X,Ch)} with Lemma \ref{lem:Xe}, we deduce the model structure $\mathcal{M}_{dg\mathcal{X}}$  follows directly from the correspondence between model structure and complete cotorsion pairs in exact categories; see \cite[Theorem 3.3]{Gil11} for details.

By a fundamental result on model categories (see, for example, \cite[Theorem 1.2.10]{Hov99}), we have ${\rm Ho}(\mathcal{M}_{dg\mathcal{X}}) \simeq \mathcal{E}\text{-}{\rm dg}\widetilde{\mathcal{X}}/ \widetilde{\mathcal{X}}_\mathcal{E}$. By Lemma \ref{lem:Xe}(2), we infer that $\mathcal{E}\text{-}{\rm dg}\widetilde{\mathcal{X}}/ \widetilde{\mathcal{X}}_\mathcal{E} = {\bf K}_{\mathcal{E}\text{-}{\rm dg}}(\mathcal{X})$. It remains to prove ${\bf K}_{\mathcal{E}\text{-}{\rm dg}}(\mathcal{X}) \simeq {\bf D}_\mathcal{X}(\mathcal{A})$.

There is a natural composite functor $F: {\bf K}_{\mathcal{E}\text{-}{\rm dg}}(\mathcal{X})\hookrightarrow {\bf K}(\mathcal{A}) \rightarrow {\bf D}_\mathcal{X}(\mathcal{A})$, where ${\bf K}(\mathcal{A}) \rightarrow {\bf D}_\mathcal{X}(\mathcal{A})$ is the quotient functor. It is clear that $F$ is a triangle functor, and it follows immediately from Lemma \ref{lem:dgX-ap} that the functor $F$ is dense.

We claim that for any complex $X\in \mathcal{E}\text{-}{\rm dg}\widetilde{\mathcal{X}}$ and any $\mathcal{X}$-quasi-isomorphism $g: M\rightarrow X$, there is a chain map $h: X\rightarrow M$ such that $gh\sim {\rm Id}_X$. For any $\mathcal{X}$-quasi-isomorphism $f: M\rightarrow N$, it is clear that ${\rm Con}(f) \in \widetilde{\mathcal{E}}$. Then the Hom-complex ${\rm Hom}_{\mathcal{A}}(X, {\rm Con}(f)) \cong {\rm Con}({\rm Hom}_{\mathcal{A}}(X, f))$ is acyclic. This yields a quasi-isomorphism ${\rm Hom}_{\mathcal{A}}(X, f)$, and moreover, an isomorphism of abelian groups ${\rm Hom}_{{\bf K}(\mathcal{A})}(X, M)\cong {\rm Hom}_{{\bf K}(\mathcal{A})}(X, N)$. Especially, let $N = X$ and we consider the isomorphism induced by the $\mathcal{X}$-quasi-isomorphism $g: M\rightarrow X$. Then for ${\rm Id}_X \in {\rm Hom}_{{\bf K}(\mathcal{A})}(X, X)$, there exists a map $h\in {\rm Hom}_{{\bf K}(\mathcal{A})}(X, M)$ such that $gh = {\rm Id}_X$ in ${\rm Hom}_{{\bf K}(\mathcal{A})}(X, X)$, i.e., $gh\sim {\rm Id}_X$.

By the above claim, we can follow a standard argument to prove that for any $X\in \mathcal{E}\text{-}{\rm dg}\widetilde{\mathcal{X}}$ and any complex $C$, the natural map $\varphi: {\rm Hom}_{{\bf K}(\mathcal{A})}(X, C)\rightarrow {\rm Hom}_{{\bf D}_\mathcal{X}(\mathcal{A})}(X, C)$ given by $f\rightarrow f/{\rm Id}_X$ is an isomorphism, where the right fraction $f/{\rm Id}_X$ denotes a map in ${\bf D}_\mathcal{X}(\mathcal{A})$. Indeed, if $\varphi(f) = 0$, then there is an $\mathcal{X}$-quasi-isomorphism $g: M\rightarrow X$ such that $fg \sim 0$. Moreover, we infer from the above argument that there is a map $h: X\rightarrow M$ such that $gh\sim {\rm Id}_X$. Then $f\sim f{\rm Id}_X\sim fgh \sim 0$, and hence $\varphi$ is a monomorphism. Any right fraction $\alpha/g \in {\rm Hom}_{{\bf D}_\mathcal{X}(\mathcal{A})}(X, C)$ can be represented by a diagram $X\stackrel{g}\longleftarrow M \stackrel{\alpha}\longrightarrow C$, where $g$ is an $\mathcal{X}$-quasi-isomorphism. There is an $\mathcal{X}$-quasi-isomorphism $h: X\rightarrow M$ such that $gh\sim {\rm Id}_X$, and moreover, we have $\alpha/g = \alpha h/{\rm Id}_X = \varphi(\alpha h)$. This implies that $\varphi$ is an epimorphism.

In particular, for any complexes $X, X'\in \mathcal{E}\text{-}{\rm dg}\widetilde{\mathcal{X}}$, we have an isomorphism induced by the functor $F$:
$${\rm Hom}_{{\bf K}_{\mathcal{E}\text{-}{\rm dg}}(\mathcal{X})}(X, X') = {\rm Hom}_{{\bf K}(\mathcal{A})}(X, X') \cong {\rm Hom}_{{\bf D}_\mathcal{X}(\mathcal{A})}(X, X').$$ Hence, the functor $F:  {\bf K}_{\mathcal{E}\text{-}{\rm dg}}(\mathcal{X})\rightarrow {\bf D}_\mathcal{X}(\mathcal{A})$ is fully faithful, and consequently, it is an equivalence. This completes the proof.
\end{proof}

\begin{remark}\label{rem:M4DY}
Dually, we have the notions of left $\mathcal{Y}$-quasi-isomorphism and relative derived category ${\bf D}_\mathcal{Y}(\mathcal{A})$. It can be shown that there is also a model structure, denoted by $\mathcal{M}_{dg\mathcal{Y}} = ({\rm Ch}(\mathcal{A}), \widetilde{\mathcal{E}}, \mathcal{E}\text{-}{\rm dg}\widetilde{\mathcal{Y}})$, on the exact category ${\rm Ch}(\mathcal{A}, \mathcal{E})$. Moreover, we have triangle-equivalences
${\rm Ho}(\mathcal{M}_{dg\mathcal{Y}}) \simeq {\bf K}_{\mathcal{E}\text{-}{\rm dg}}(\mathcal{Y}) \simeq {\bf D}_\mathcal{Y}(\mathcal{A})$.
\end{remark}

\begin{remark}\label{rem:1}
Chen proved that under an additional assumption $\mathcal{X}{\text-}{\rm dim}\mathcal{A} < \infty$, the natural composite functor ${\bf K}(\mathcal{X})\hookrightarrow {\bf K}(\mathcal{A}) \rightarrow {\bf D}_\mathcal{X}(\mathcal{A})$ is a triangle-equivalence, where ${\bf K}(\mathcal{X})$ is the chain homotopy category of complexes with all items in $\mathcal{X}$; see \cite[Proposition 3.5]{Chen10}. It is clear that ${\bf K}_{\mathcal{E}\text{-}{\rm dg}}(\mathcal{X})\subseteq {\bf K}(\mathcal{X})$. Assume $\mathcal{X}{\text-}{\rm dim}\mathcal{A} < \infty$. Similar to an argument in \cite[Proposition 5.2]{Ren19}, one has  ${\bf K}_{\mathcal{E}\text{-}{\rm dg}}(\mathcal{X})\supseteq {\bf K}(\mathcal{X})$. Indeed, for any $X\in {\bf K}(\mathcal{X})$, by Lemma \ref{lem:dgX-ap} we have an $\mathcal{X}$-quasi-isomorphism $f: X'\rightarrow X$ with $X'\in {\bf K}_{\mathcal{E}\text{-}{\rm dg}}(\mathcal{X})$. Then ${\rm Con}(f)$ is in both $\widetilde{\mathcal{E}}$ and ${\bf K}(\mathcal{X})$. Since $\mathcal{X}{\text-}{\rm dim}\mathcal{A} < \infty$, one has ${\rm Z}_n{\rm Con}(f) \in \mathcal{X}$ for any $n\in \mathbb{Z}$. Then ${\rm Con}(f)\in \widetilde{\mathcal{X}}_\mathcal{E}$, and moreover, by the short exact sequence $0\rightarrow X\rightarrow {\rm Con}(f)\rightarrow \Sigma X'\rightarrow 0$ we deduce that $X\in {\bf K}_{\mathcal{E}\text{-}{\rm dg}}(\mathcal{X})$.
\end{remark}


\section{ \bf Models for chain homotopy categories}

In this section, we intend to find model structures to realize the chain homotopy categories ${\bf K}(\mathcal{X})$ and ${\bf K}(\mathcal{Y})$; see Theorem \ref{thm:M5}.

Let $\mathcal{E}\text{-}{\rm dw}\widetilde{\mathcal{X}}$ be the class of complexes $X$ for which each item $X_n\in \mathcal{X}$. The prefix ``$\mathcal{E}$'' is used to indicate that we consider the right orthogonal $(\mathcal{E}\text{-}{\rm dw}\widetilde{\mathcal{X}})^{\perp}$  with respect to ${\rm Ext}^1_{{\rm Ch}(\mathcal{E})}(-, -)$. Note that $(\mathcal{E}\text{-}{\rm dw}\widetilde{\mathcal{X}})^{\perp}\subseteq \widetilde{\mathcal{E}}$ by Proposition \ref{thm:(dgX,E)}. Let $\mathcal{E}\text{-}{\rm ac}\widetilde{\mathcal{X}} = \mathcal{E}\text{-}{\rm dw}\widetilde{\mathcal{X}} \cap \widetilde{\mathcal{E}}$ be the class of complexes $X$ which are right $\mathcal{X}$-acyclic with each item $X_n\in \mathcal{X}$.

By the proof of Lemma \ref{lem:Xe}, every complex $X$ in $\widetilde{\mathcal{X}}_\mathcal{E}$ is of the form $\bigoplus {\rm D}^nX'_n$ with $X'_n\in \mathcal{X}$, which implies $\widetilde{\mathcal{X}}_\mathcal{E} \subseteq (\mathcal{E}\text{-}{\rm dw}\widetilde{\mathcal{X}})^{\perp}$.

\begin{proposition}\label{prop:(dwX)}
$(\mathcal{E}\text{-}{\rm dw}\widetilde{\mathcal{X}}, (\mathcal{E}\text{-}{\rm dw}\widetilde{\mathcal{X}})^{\perp})$ is a hereditary cotorsion pair in ${\rm Ch}(\mathcal{A}, \mathcal{E})$.
\end{proposition}

\begin{proof}
Let $\widehat{\mathcal{S}}$ be the collection of all complexes $C$ satisfying that any chain map $X\rightarrow C$ from complex $X\in \mathcal{E}\text{-}{\rm dw}\widetilde{\mathcal{X}}$ to $C$ is null homotopic. Then, for any $X\in \mathcal{E}\text{-}{\rm dw}\widetilde{\mathcal{X}}$ and any $C\in \widehat{\mathcal{S}}$, we have
$${\rm Ext}^1_{{\rm Ch}(\mathcal{E})}(X, C) = {\rm Ext}^1_{dw}(X, C)\cong {\rm Hom}_{{\rm Ch}(\mathcal{A})}(X, \Sigma C)/\sim = 0.$$
This implies that $(\mathcal{E}\text{-}{\rm dw}\widetilde{\mathcal{X}})^{\perp} = \widehat{\mathcal{S}}$ and $\mathcal{E}\text{-}{\rm dw}\widetilde{\mathcal{X}} \subseteq {^{\perp}\widehat{\mathcal{S}}}$.

It is clear that for any $A\in \mathcal{A}$, the complex ${\rm D}^nA$ is contractible, and is in $\widehat{\mathcal{S}}$. Analogous to \cite[Lemma 3.1(6)]{Gil04}, one can prove that for any complex $X$, there is a natural isomorphism
${\rm Ext}^1_{{\rm Ch}(\mathcal{E})}(X, {\rm D}^{n+1}A) \cong {\rm Ext}^1_{\mathcal{E}}(X_n, A)$. Hence, for any $X\in {^{\perp}\widehat{\mathcal{S}}}$, we have ${\rm Ext}^1_{\mathcal{E}}(X_n, A) = 0$. By Lemma \ref{admissiblebalancedpairs}(2) this implies $X_n\in \mathcal{X}$ for each $n$, and then $X\in \mathcal{E}\text{-}{\rm dw}\widetilde{\mathcal{X}}$. Hence, $(\mathcal{E}\text{-}{\rm dw}\widetilde{\mathcal{X}}, (\mathcal{E}\text{-}{\rm dw}\widetilde{\mathcal{X}})^{\perp})$ is a cotorsion pair in ${\rm Ch}(\mathcal{A}, \mathcal{E})$.

The heredity is easy to check from the fact that $\mathcal{X}$ is closed under taking kernels of admissible epimorphisms.
\end{proof}

In order to show that the cotorsion pair $(\mathcal{E}\text{-}{\rm dw}\widetilde{\mathcal{X}}, (\mathcal{E}\text{-}{\rm dw}\widetilde{\mathcal{X}})^{\perp})$ is complete, we need the following facts.


\begin{lemma}\label{pushoutinexactcategory}
Consider a commutative square in an exact category $(\mathcal{A}, \mathcal{E})$
$$\xymatrix@C=22pt@R=20pt{ A\ \ar@{>->}[r]^{i} \ar[d]_{f} &B \ar[d]^{f'} \\
A'\ \ar@{>->}[r]^{i'} &B'
  }$$
in which $i$ and $i'$ are admissible monomorphisms. Then the square is a pushout if and only if the sequence
$$A\xrightarrow{\ \tiny\begin{pmatrix}i\\-f\end{pmatrix}\ }B\oplus A'\xrightarrow{\tiny\begin{pmatrix}f'&i'\end{pmatrix}} B'$$is in $\mathcal{E}$.
\end{lemma}

\begin{lemma}\label{dwboundedisdg}
For any bounded below complex $X \in \mathcal{E}\text{-}{\rm dw}\widetilde{\mathcal{X}}$ , one has $X\in \mathcal{E}\text{-}{\rm dg}\widetilde{\mathcal{X}}$.
\end{lemma}

\begin{proof}
Without loss of generality, we may assume that $(X, d) = \cdots \rightarrow X_1\rightarrow X_0\rightarrow 0$, and a chain map $f$ from $(X, d)$ to $(E, e)$
$$\xymatrix@C=22pt@R=20pt{ X =\ar[d]_{f} &\cdots \ar[r] & X_{n}\ar[r]^{d_n\ \ }\ar[d]^{f_n} &X_{n-1}\ar[r]^{d_{n-1}}\ar[d]^{f_{n-1}} &X_{n-2}\ar[r]\ar[d]^{f_{n-2}} &\cdots \ar[r] & X_0 \ar[r]^{d_0} \ar[d]^{f_0} & 0 \ar[d] \ar[r] &\cdots\\
E = &\cdots \ar[r]&E_n \ar[r]^{e_n\ \ }\ar[r] &E_{n-1} \ar[r]^{e_{n-1}}\ar[r] &E_{n-2} \ar[r]\ar[r] & \cdots \ar[r] & E_0 \ar[r]^{e_0} &E_{-1} \ar[r] &\cdots }$$
with $E\in \widetilde{\mathcal{E}}$. We need to construct the chain homotopy $\{s_i: X_i\rightarrow E_{i+1}\}$ by setting $s_i =0$ for $i\leq -1$ and induction. Now we assume that $s_i$ is constructed for $i\leq n-2$.

Let  $f'_{n-1} = f_{n-1}-s_{n-2}d_{n-1}$. Then $e_{n-1}f'_{n-1} = 0$, thus $X_{n-1}\stackrel{f'_{n-1}} \rightarrow E_{n-1}$ factor through ${\rm Z}_{n-1}E$. Since $0\rightarrow {\rm Z}_{n}E\rightarrow E_{n}\rightarrow {\rm Z}_{n-1}E\rightarrow 0$ is right $\mathcal{X}$-acyclic, we have
$$\xymatrix{  & & &X_{n-1}\ar[d] \ar@{-->}[dl]_{s_{n-1}} &\\
0 \ar[r] &{\rm Z}_{n}E\ar[r] & E_{n}\ar[r] & {\rm Z}_{n-1}E \ar[r] &0, }$$
which implies $e_ns_{n-1} = f'_{n-1} =  f_{n-1}-s_{n-2}d_{n-1}$ and thus the chain map $f$ is null homotopic.
\end{proof}

\begin{lemma}\label{limit}
Let $K^{0}\stackrel{i^{0}}\longrightarrow K^{1}\stackrel{i^{1}}\longrightarrow K^{2}\longrightarrow \cdots$ be a direct system in ${\rm Ch}(\mathcal{A}, \mathcal{E})$ with each morphism being admissible monic and $K =\varinjlim_{n\geq 0} K^{n}$. If each $K^{n}$ is right $\mathcal{X}$-acyclic, then $K$ is, as well.

\end{lemma}

\begin{proof}
From the assumption we have the degree-wise right $\mathcal{X}$-acyclic (equivalently, degree-wise left $\mathcal{Y}$-acyclic) sequences $0\rightarrow K^{n}\rightarrow K^{n+1}\rightarrow K^{n+1}/K^{n}\rightarrow 0$. Then for all $Y \in \mathcal{Y}$ we have an inverse system of epimorphisms
$$\cdots \rightarrow {\rm Hom}_{\mathcal{A}}( K^{2}, Y)\rightarrow {\rm Hom}_{\mathcal{A}}( K^{1}, Y)\rightarrow {\rm Hom}_{\mathcal{A}}( K^{0},Y).$$
Since $K^{n}$ is left $\mathcal{Y}$-acyclic for all $n$, it follows from \cite[Theorem 1.6.13]{EJ00} that
${\rm Hom}_{\mathcal{A}}( K, Y) \cong \varprojlim_{n\geq 0} {\rm Hom}_{\mathcal{A}}( K^{n}, Y)$ is acyclic, as required.
\end{proof}

\begin{lemma}\label{lem:dwX-ap4>0}
For any bounded below complex $C$, there is a short exact sequence $0\rightarrow K\rightarrow X\rightarrow C\rightarrow 0$ in ${\rm Ch}(\mathcal{A}, \mathcal{E})$, where $X\in \mathcal{E}\text{-}{\rm dw}\widetilde{\mathcal{X}}$, $K\in (\mathcal{E}\text{-}{\rm dw}\widetilde{\mathcal{X}})^{\perp}$ and both $X$ and $K$ are bounded below.
\end{lemma}

\begin{proof}
Without loss of generality, we may assume that $C = \cdots \rightarrow C_1\rightarrow C_0\rightarrow 0$. Denote by $C_{[n,0]} = 0\rightarrow C_n\rightarrow \cdots \rightarrow C_1\rightarrow C_0\rightarrow 0$ the left brutal truncation of $C$ at degree $n$. Then there is a direct system
$C_{[0,0]}\rightarrow C_{[1,0]}\rightarrow  \cdots \rightarrow C_{[n,0]}\rightarrow \cdots
$, such that
$C =\varinjlim_{n\geq 0} C_{[n,0]}$.

We first prove by induction that the assertion holds for any bounded complex $C_{[n,0]}$. Let $n = 0$. For $C_0\in \mathcal{A}$, there is an $\mathcal{X}$-resolution $\cdots \rightarrow X_1\rightarrow X_0\rightarrow C_0\rightarrow 0$. Let
$K_0 = {\rm Ker}(X_0\rightarrow C_0)$. Then, there is a short exact sequence of complexes  $0\rightarrow K^{(0)}\rightarrow X^{(0)} \stackrel{\pi^{(0)}}\rightarrow C_{[0,0]}\rightarrow 0$, where $C_{[0,0]} = {\rm S}^0C_0$,  $X^{(0)}$ is the deleted $\mathcal{X}$-resolution $\cdots \rightarrow X_2\rightarrow X_1 \rightarrow X_0\rightarrow 0$ of $C_0$, and $K^{(0)} = \cdots\rightarrow X_2\rightarrow X_1\rightarrow K_0\rightarrow 0$. We note that the negative items of $X^{(0)}$ and $K^{(0)}$ are zero, $X^{(0)}\in \mathcal{E}\text{-}{\rm dw}\widetilde{\mathcal{X}}$, and $K^{(0)}$ is right $\mathcal{X}$-acyclic. It follows from \cite[Lemma 2.4]{CFH06} that for any complex $X'\in \mathcal{E}\text{-}{\rm dw}\widetilde{\mathcal{X}}$, the Hom-complex ${\rm Hom}_{\mathcal{A}}(X', K^{(0)})$ is acyclic. Then we have
${\rm Ext}^1_{{\rm Ch}(\mathcal{E})}(X', K^{(0)}) = {\rm Ext}^1_{dw}(X', K^{(0)}) \cong {\rm H}_{-1}{\rm Hom}_{\mathcal{A}}(X', K^{(0)}) = 0$, i.e., $K^{(0)}\in (\mathcal{E}\text{-}{\rm dw}\widetilde{\mathcal{X}})^{\perp}$. It follows from the construction that $\pi^{(0)}$ is a right $\mathcal{E}\text{-}{\rm dw}\widetilde{\mathcal{X}}$-approximation.

Now we assume that the assertion is true for $C_{[n,0]}$, i.e., there is a short exact sequence  $0\rightarrow K^{(n)}\rightarrow X^{(n)}\stackrel{\pi^{(n)}}\rightarrow C_{[n,0]}\rightarrow 0$ in ${\rm Ch}(\mathcal{A}, \mathcal{E})$, where $X^{(n)}\in \mathcal{E}\text{-}{\rm dw}\widetilde{\mathcal{X}}$ and $K^{(n)}\in (\mathcal{E}\text{-}{\rm dw}\widetilde{\mathcal{X}})^{\perp}$ with both zero negative items. We consider an $\mathcal{X}$-resolution  $\cdots \rightarrow P_1\rightarrow P_0\stackrel{\varepsilon}\rightarrow C_{n+1}\rightarrow 0$ of the object $C_{n+1}$. Then $\varepsilon: P_0\rightarrow C_{n+1}$ induces a chain map $P\rightarrow {\rm S}^nC_{n+1}$, denoted also by $\varepsilon$, where $P$ is obtained by setting each $P_i$ at degree $n+i$. Similar to the above argument, $P$ is a right $\mathcal{E}\text{-}{\rm dw}\widetilde{\mathcal{X}}$-approximation of ${\rm S}^nC_{n+1}$ and ${\rm Ker}\varepsilon \in (\mathcal{E}\text{-}{\rm dw}\widetilde{\mathcal{X}})^{\perp}$. It is clear that $d_{n+1}^C: C_{n+1}\rightarrow C_n$ induces a chain map $d: {\rm S}^nC_{n+1}\rightarrow C_{[n,0]}$. As $\pi^{(n)}: X^{(n)}\rightarrow C_{[n,0]}$ is a right $\mathcal{E}\text{-}{\rm dw}\widetilde{\mathcal{X}}$-approximation, for $d\varepsilon: P\rightarrow C_{[n,0]}$, there exists a chain map $f: P\rightarrow X^{(n)}$ such that the following diagram commutes
 $$\xymatrix@C=22pt@R=20pt{ P\ \ar[d]^{\varepsilon} \ar[r]^{f} & X^{(n)}\ar[d]^{\pi^{(n)}} \\
{\rm S}^nC_{n+1}\ \ar[r]^{d} &C_{[n,0]}.
  }$$Set $X^{(n+1)} = {\rm Con}(f)$. Then the negative items of $X^{(n+1)}$ vanish and $X^{(n+1)}\in \mathcal{E}\text{-}{\rm dw}\widetilde{\mathcal{X}}$. The above diagram induces the following commutative diagram
  $$\xymatrix@C=22pt@R=20pt{ &0  \ar[d] &0\ar[d] &0 \ar[d] \\
  0\ar[r] &K^{(n)} \ar[r]^{g^{(n)}} \ar[d] &K^{(n+1)} \ar[r]\ar[d] &\Sigma\mathrm{Ker}\varepsilon \ar[r]^{} \ar[d] &0\\
  0\ar[r] &X^{(n)} \ar[r] \ar[d]^{\pi^{(n)}} &X^{(n+1)} \ar[r]\ar[d]^{\pi^{(n+1)}} &\Sigma P \ar[r]^{} \ar[d]^{\Sigma\varepsilon} &0\\
0\ar[r] &C_{[n,0]}\ar[d]\ar[r] &C_{[n+1,0]}\ar[d] \ar[r] &\Sigma{\rm S}^{n}C_{n+1}\ar[d] \ar[r] &0 \\
 &0  &0 &\;0. }$$By the third column, ${\rm Coker}g^{(n)} \cong\Sigma\mathrm{Ker}\varepsilon\in (\mathcal{E}\text{-}{\rm dw}\widetilde{\mathcal{X}})^{\perp}$. Since $K^{(n)}\in (\mathcal{E}\text{-}{\rm dw}\widetilde{\mathcal{X}})^{\perp}$, it follows from the first row that $K^{(n+1)}\in (\mathcal{E}\text{-}{\rm dw}\widetilde{\mathcal{X}})^{\perp}$ with vanished negative items.

Note that all columns are right $\mathcal{X}$-acyclic and each $C_{[i,0]}$ is the left brutal truncation of $C$. By Lemma \ref{lem:directlimit} we get the required short exact sequence
$$0\longrightarrow K \longrightarrow X \longrightarrow C \longrightarrow 0$$
in ${\rm Ch}(\mathcal{A}, \mathcal{E})$, where $X = \varinjlim_{n\geq 0} X^{(n)}$ and $K = \varinjlim_{n\geq 0} K^{(n)}$ are bounded below. It follows from Lemma \ref{admissiblebalancedpairs}(3), Lemma \ref{limit} and \cite[Lemma 2.4]{CFH06} that $X\in \mathcal{E}\text{-}{\rm dw}\widetilde{\mathcal{X}}$ and $K\in (\mathcal{E}\text{-}{\rm dw}\widetilde{\mathcal{X}})^{\perp}$.
\end{proof}

\begin{lemma}\label{lem:dwX-ap}
For any complex $(C, d)$, there exists a short exact sequence $0\rightarrow K\rightarrow X\rightarrow C\rightarrow 0$ in ${\rm Ch}(\mathcal{A}, \mathcal{E})$, where $X\in \mathcal{E}\text{-}{\rm dw}\widetilde{\mathcal{X}}$ and $K$ is right $\mathcal{X}$-acyclic. Moreover, if $(\mathcal{E}\text{-}{\rm dw}\widetilde{\mathcal{X}})^{\perp}$ is closed under direct sums, then $K\in (\mathcal{E}\text{-}{\rm dw}\widetilde{\mathcal{X}})^{\perp}$.
\end{lemma}

\begin{proof}
Let $C_{\geq n} = \cdots \rightarrow C_{n+2}\rightarrow C_{n+1}\rightarrow {\rm Ker}d_n\rightarrow 0$. Then the natural chain map $\iota_{n}: C_{\geq n}\rightarrow C_{\geq n-1}$ yields a  direct system $(C_{\geq n}, \iota_{n})$, and $C = \varinjlim_{n\leq 0} C_{\geq n}$.

By Lemma \ref{lem:dwX-ap4>0} we have a chain of short exact sequences in ${\rm Ch}(\mathcal{A}, \mathcal{E})$

$$\xymatrix@C=22pt@R=20pt{
 K'_{(0)} \ar@{-->}[r]^{j'_0\ } \ar@{>->}[d]  &K'_{(-1)} \ar@{>->}[d]\ar@{-->}[r]^{j'_{-1}\ } & K'_{(-2)} \ar@{-->}[r]^{j'_{-2}} \ar@{>->}[d]  & \cdots\\
 X'_{(0)} \ar@{->>}[d]^{c_0} \ar@{-->}[r]^{i'_0\ } & X'_{(-1)} \ar@{-->}[r]^{i'_{-1}\ } \ar@{->>}[d]^{c_{-1}} & X'_{(-2)} \ar@{->>}[d]^{c_{-2}}  \ar@{-->}[r]^{i'_{-1}} &\cdots \\
 C_{\geq 0} \ar[r]^{\iota_{0}\ } & C_{\geq -1} \ar[r]^{\iota_{-1}} & C_{\geq -2}  \ar[r]^{\iota_{-2}} & \cdots.
  }$$
It follows from the proof of Lemma \ref{lem:dwX-ap4>0} that all $X'_{(n)}$ are bounded below in $\mathcal{E}\text{-}{\rm dw}\widetilde{\mathcal{X}}$ and all $c_{n}$ are right $\mathcal{E}\text{-}{\rm dw}\widetilde{\mathcal{X}}$-approximations.
Thus there exists $ i'_n : X'_{(n)} \rightarrow X'_{(n-1)}$ and $j'_n : K'_{(n)} \rightarrow K'_{(n-1)}$ such that the diagram is commutative. It follows from Lemma \ref{dwboundedisdg} that $X'_{(n)}\in \mathcal{E}\text{-}{\rm dg}\widetilde{\mathcal{X}}$.
By Proposition \ref{thm:(dgX,E)}, $X'_{(0)}$ admits a short exact sequence in ${\rm Ch}(\mathcal{A}, \mathcal{E})$
$$\xymatrix{X'_{(0)}\ \ar@{>->}[r] & Q_{0}\ar@{->>}[r] & L_{0}\\}$$
with $L_{0}\in \mathcal{E}\text{-}{\rm dg}\widetilde{\mathcal{X}}$, and $Q_{0} \in \mathcal{E}\text{-}{\rm dg}\widetilde{\mathcal{X}} \cap \widetilde{\mathcal{E}} = \widetilde{\mathcal{X}}_\mathcal{E}$ by Lemma \ref{lem:Xe}. It follows from the pushout diagram
$$\xymatrix@C=22pt@R=20pt{ X'_{(0)}\ \ar@{>->}[r]^{a_0} \ar[d]^{-i'_0} &Q_{0}\ \ar[d]^{f_0}\ar@{->>}[r] & L_{0} \ar@{=}[d]\\
X'_{(-1)}\ \ar@{>->}[r]^{\alpha_0} &U_{0}\ \ar@{->>}[r] & L_{0}
}$$
together with Lemma \ref{pushoutinexactcategory} that the sequence $$\xymatrix{X'_{(0)}\ \ar@{>->}[r]^{\tiny\begin{pmatrix}a_0 \\i'_0\end{pmatrix}\ \ \ \ \ \ } & Q_{0} \oplus X'_{(-1)}\ \ar@{->>}[r]^{\ \ \ (f_0 \ \alpha_0)} & U_{0}\\}$$
is exact in ${\rm Ch}(\mathcal{A}, \mathcal{E})$. Since $X'_{(-1)}, L_{0} \in \mathcal{E}\text{-}{\rm dg}\widetilde{\mathcal{X}}$, $U_{0}\in \mathcal{E}\text{-}{\rm dg}\widetilde{\mathcal{X}} \subseteq \mathcal{E}\text{-}{\rm dw}\widetilde{\mathcal{X}}$, the admissible monomorphism $\begin{pmatrix}a_0 \\i'_0\end{pmatrix}$ becomes degree-wise split. We get the following commutative diagram
$$\xymatrix@C=35pt@R=23pt{
 K'_{(0)} \ar@{-->}[r]^{\begin{pmatrix}b_0 \\j'_0\end{pmatrix}\ \ \ \ } \ar@{>->}[d]  &Q_{0} \oplus K'_{(-1)} \ar@{>->}[d]\ar@{-->}[r] & Q_{0} \oplus K'_{(-2)} \ar@{-->}[r] \ar@{>->}[d]  & \cdots\\
 X'_{(0)} \ar@{->>}[d]^{c_0} \ar@{>->}[r]^{\begin{pmatrix}a_0 \\i'_0\end{pmatrix}\ \ \ \ } & Q_{0} \oplus X'_{(-1)} \ar@{->}[r]^{\begin{pmatrix}1 \ \ 0 \\ 0 \ \ i'_{-1}\end{pmatrix}} \ar@{->>}[d]^{(0, c_{-1})} & Q_{0} \oplus X'_{(-2)} \ar@{->>}[d]^{(0, c_{-2})}  \ar@{-->}[r] &\cdots \\
 C_{\geq 0} \ar[r]^{\iota_{0}} & C_{\geq -1} \ar[r]^{\iota_{-1}} & C_{\geq -2}  \ar[r]^{\iota_{-2}} & \cdots
  }$$
where all the columns are exact in ${\rm Ch}(\mathcal{A}, \mathcal{E})$, and
 $j_0=\begin{pmatrix}b_0 \\j'_0\end{pmatrix}$ is admissible monic with $b_0$ being the restriction of $a_0$ to $K'_{(0)}$. Since $Q_{0} \in \mathcal{E}\text{-}{\rm dg}\widetilde{\mathcal{X}} \cap \widetilde{\mathcal{E}} = \widetilde{\mathcal{X}}_\mathcal{E} \subseteq (\mathcal{E}\text{-}{\rm dw}\widetilde{\mathcal{X}})^{\perp}$, we have $Q_{0} \oplus X'_{(n)}\in \mathcal{E}\text{-}{\rm dg}\widetilde{\mathcal{X}}$ and $Q_{0} \oplus K'_{(n)} \in (\mathcal{E}\text{-}{\rm dw}\widetilde{\mathcal{X}})^{\perp}$ for all $n\leq -1$.
Repeating the above process we get a commutative diagram of complexes and renamed by
$$\xymatrix@C=22pt@R=20pt{
 K_{(0)} \ar@{>->}[r]^{j_0\ } \ar@{>->}[d]  &K_{(-1)} \ar@{>->}[d]\ar@{>->}[r]^{j_{-1}} & K_{(-2)} \ar@{>->}[r]^{j_{-2}} \ar@{>->}[d]  & \cdots\\
 X_{(0)} \ar@{->>}[d]^{c_0} \ar@{>->}[r]^{i_0\ } & X_{(-1)} \ar@{>->}[r]^{i_{-1}} \ar@{->>}[d]^{c_{-1}} & X_{(-2)} \ar@{->>}[d]^{c_{-2}}  \ar@{>->}[r]^{i_{-1}} &\cdots \\
 C_{\geq 0} \ar[r]^{\iota_{0}\ } & C_{\geq -1} \ar[r]^{\iota_{-1}} & C_{\geq -2}  \ar[r]^{\iota_{-2}} & \cdots,
  }$$
where all the columns are exact in ${\rm Ch}(\mathcal{A}, \mathcal{E})$, $j_n$ is admissible monic, $i_n$ is monic and degree-wise split, $K_{(n)}\in (\mathcal{E}\text{-}{\rm dw}\widetilde{\mathcal{X}})^{\perp}\subseteq \widetilde{\mathcal{E}}$  and $X_{(n)}\in \mathcal{E}\text{-}{\rm dw}\widetilde{\mathcal{X}}$ for $n\leq 0$.

Denote the limits of direct systems $(X_{(n)}, i_n)$ and $(K_{(n)}, j_n)$ by $X = \varinjlim_{n\leq 0} X_{(n)}$ and $K = \varinjlim_{n\leq 0} K_{(n)}$, respectively, we get the desired exact sequence $0\rightarrow K\rightarrow X\rightarrow C \rightarrow 0$ which is degree-wise left $\mathcal{Y}$-acyclic by Lemma \ref{lem:directlimit}. Then by Lemmas \ref{admissiblebalancedpairs} and \ref{limit}, $X \in \mathcal{E}\text{-}{\rm dw}\widetilde{\mathcal{X}}$ and $K$ is right $\mathcal{X}$-acyclic.

Since each $j_n$ is admissible monic, by applying ${\rm Hom}_{\mathcal{A}}(-,\mathcal{Y})$ we obtain that the short exact sequence $0\rightarrow \bigoplus K_{(n)} \rightarrow \bigoplus K_{(n)}\rightarrow K\longrightarrow 0$ is exact in ${\rm Ch}(\mathcal{A}, \mathcal{E})$. Together with the  assumption that $(\mathcal{E}\text{-}{\rm dw}\widetilde{\mathcal{X}})^{\perp}$ is closed under direct sums and the fact that $(\mathcal{E}\text{-}{\rm dw}\widetilde{\mathcal{X}}, (\mathcal{E}\text{-}{\rm dw}\widetilde{\mathcal{X}})^{\perp})$ is hereditary, we have $K\in (\mathcal{E}\text{-}{\rm dw}\widetilde{\mathcal{X}})^{\perp}$.
\end{proof}

\begin{proposition}\label{prop:(dwX,+)}
If $(\mathcal{E}\text{-}{\rm dw}\widetilde{\mathcal{X}})^{\perp}$ is closed under direct sums, then the cotorsion pair $(\mathcal{E}\text{-}{\rm dw}\widetilde{\mathcal{X}}, (\mathcal{E}\text{-}{\rm dw}\widetilde{\mathcal{X}})^{\perp})$ is complete.
\end{proposition}

\begin{proof}
Part of the completeness has been proved in Lemma \ref{lem:dwX-ap}. We will use a standard argument (known as Salce's trick) to prove the other part.

Dual to Proposition \ref{prop:(X,Ch)}, one can prove that for any complex $C$, there is a short exact sequence $0\rightarrow C\rightarrow Y\rightarrow L\rightarrow 0$ in ${\rm Ch}(\mathcal{A}, \mathcal{E})$, where $Y\in \widetilde{\mathcal{Y}}_\mathcal{E}$. For the complex $L$, we infer from Lemma \ref{lem:dwX-ap} that there is a short exact sequence $0\rightarrow K\rightarrow X\rightarrow L\rightarrow 0$ in ${\rm Ch}(\mathcal{A}, \mathcal{E})$, where $X\in \mathcal{E}\text{-}{\rm dw}\widetilde{\mathcal{X}}$ and $K\in (\mathcal{E}\text{-}{\rm dw}\widetilde{\mathcal{X}})^{\perp}$. Consider the following pullback of $X\rightarrow L$ and $Y\rightarrow L$:
$$\xymatrix@C=22pt@R=20pt{ & & 0\ar[d] & 0\ar[d] \\
& & K \ar@{=}[r]^{} \ar[d]  &K \ar[d]\\
0 \ar[r] & C \ar@{=}[d] \ar[r] &D \ar@{-->}[r] \ar@{-->}[d] & X \ar[r] \ar[d] &0 \\
0 \ar[r] & C \ar[r] & Y \ar[r]^{} \ar[d] & L \ar[r] \ar[d]& 0\\
 & & 0 & \;0.
  }$$
Note that $\widetilde{\mathcal{Y}}_\mathcal{E} \subseteq (\mathcal{E}\text{-}{\rm dw}\widetilde{\mathcal{X}})^{\perp}$ since the complexes in $\widetilde{\mathcal{Y}}_\mathcal{E}$ are contractible. We infer from the middle column that $D\in (\mathcal{E}\text{-}{\rm dw}\widetilde{\mathcal{X}})^{\perp}$ since so are $Y$ and $K$.  Then, for any complex $C$, there is a short exact sequence
$0\rightarrow C\rightarrow D\rightarrow X\rightarrow 0$, where $X\in \mathcal{E}\text{-}{\rm dw}\widetilde{\mathcal{X}}$ and $D\in (\mathcal{E}\text{-}{\rm dw}\widetilde{\mathcal{X}})^{\perp}$. Now, together with Lemma \ref{lem:dwX-ap}, we infer that the cotorsion pair $(\mathcal{E}\text{-}{\rm dw}\widetilde{\mathcal{X}}, (\mathcal{E}\text{-}{\rm dw}\widetilde{\mathcal{X}})^{\perp})$ is complete.
\end{proof}

\begin{lemma}\label{lem:2Xe}
$\mathcal{E}\text{-}{\rm dw}\widetilde{\mathcal{X}} \cap (\mathcal{E}\text{-}{\rm dw}\widetilde{\mathcal{X}})^{\perp} = \widetilde{\mathcal{X}}_\mathcal{E}$.
\end{lemma}

\begin{proof}
Analogous to the proof of Lemma \ref{lem:Xe}, we have $\widetilde{\mathcal{X}}_\mathcal{E} \subseteq \mathcal{E}\text{-}{\rm dw}\widetilde{\mathcal{X}} \cap (\mathcal{E}\text{-}{\rm dw}\widetilde{\mathcal{X}})^{\perp}$.
For any $X\in \mathcal{E}\text{-}{\rm dw}\widetilde{\mathcal{X}} \cap (\mathcal{E}\text{-}{\rm dw}\widetilde{\mathcal{X}})^{\perp}$, as ${\rm Id}_X\sim 0$, we have $X \in \widetilde{\mathcal{X}}_\mathcal{E}$, as desired.
\end{proof}


Under mild conditions, we have model structures for the chain homotopy categories as stated below.


\begin{theorem}\label{thm:M5}
The following hold on the exact category ${\rm Ch}(\mathcal{A}, \mathcal{E})$:
\begin{enumerate}
\item
If the subcategory $(\mathcal{E}\text{-}{\rm dw}\widetilde{\mathcal{X}})^{\perp}$ is closed under direct sums, then  $\mathcal{M}_{dw\mathcal{X}} = (\mathcal{E}\text{-}{\rm dw}\widetilde{\mathcal{X}}, (\mathcal{E}\text{-}{\rm dw}\widetilde{\mathcal{X}})^{\perp}, {\rm Ch}(\mathcal{A}))$ is a model structure with triangle-equivalence ${\rm Ho}(\mathcal{M}_{dw\mathcal{X}}) \simeq {\bf K}(\mathcal{X})$;
\item
If the subcategory ${}^{\perp}(\mathcal{E}\text{-}{\rm dw}\widetilde{\mathcal{Y}})$ is closed under direct products, then $\mathcal{M}_{dw\mathcal{Y}} = ({\rm Ch}(\mathcal{A}), {^{\perp}(\mathcal{E}\text{-}{\rm dw}\widetilde{\mathcal{Y}})}, \mathcal{E}\text{-}{\rm dw}\widetilde{\mathcal{Y}})$ is a model structure with triangle-equivalence
 ${\rm Ho}(\mathcal{M}_{dw\mathcal{Y}}) \simeq {\bf K}(\mathcal{Y})$.
\end{enumerate}
In particular, in the case of $(\mathcal{E}\text{-}{\rm dw}\widetilde{\mathcal{X}})^{\perp}= {}^{\perp}(\mathcal{E}\text{-}{\rm dw}\widetilde{\mathcal{Y}})$, there are  model structures $\mathcal{M}_{dw\mathcal{X}} = (\mathcal{E}\text{-}{\rm dw}\widetilde{\mathcal{X}}, (\mathcal{E}\text{-}{\rm dw}\widetilde{\mathcal{X}})^{\perp}, {\rm Ch}(\mathcal{A}))$ and  $\mathcal{M}_{dw\mathcal{Y}} = ({\rm Ch}(\mathcal{A}), {^{\perp}(\mathcal{E}\text{-}{\rm dw}\widetilde{\mathcal{Y}})}, \mathcal{E}\text{-}{\rm dw}\widetilde{\mathcal{Y}})$ on the exact category ${\rm Ch}(\mathcal{A}, \mathcal{E})$.
\end{theorem}

\begin{proof}
We claim that $(\mathcal{E}\text{-}{\rm dw}\widetilde{\mathcal{X}})^{\perp}$ is a thick subcategory under these assumptions. Since the cotorsion pair $(\mathcal{E}\text{-}{\rm dw}\widetilde{\mathcal{X}}, (\mathcal{E}\text{-}{\rm dw}\widetilde{\mathcal{X}})^{\perp})$ is hereditary, it suffices to prove that $(\mathcal{E}\text{-}{\rm dw}\widetilde{\mathcal{X}})^{\perp}$ is closed under taking kernels of admissible epimorphisms. That is, for any short exact sequence $0\rightarrow K\rightarrow C\stackrel{\alpha}\rightarrow D\rightarrow 0$ in ${\rm Ch}(\mathcal{A}, \mathcal{E})$ for which $C, D\in (\mathcal{E}\text{-}{\rm dw}\widetilde{\mathcal{X}})^{\perp}$, we need to show that $K\in (\mathcal{E}\text{-}{\rm dw}\widetilde{\mathcal{X}})^{\perp}$. It follows from $C, D\in (\mathcal{E}\text{-}{\rm dw}\widetilde{\mathcal{X}})^{\perp}\subseteq \widetilde{\mathcal{E}}$ that $K \in \widetilde{\mathcal{E}}$. It suffices to prove that for any $X\in \mathcal{E}\text{-}{\rm dw}\widetilde{\mathcal{X}}$, the morphism ${\rm Hom}_{{\rm Ch}(\mathcal{A})}( X, C)\rightarrow {\rm Hom}_{{\rm Ch}(\mathcal{A})}( X,D)$ induced by $\alpha$ is surjective.

For any chain map $f: X\rightarrow D$, by Proposition \ref{prop:(dwX,+)}, for $X$ there is a short exact sequence $0\rightarrow X\stackrel{\beta}\rightarrow M\rightarrow X'\rightarrow 0$ in ${\rm Ch}(\mathcal{A}, \mathcal{E})$, where $M\in\mathcal{E}\text{-}{\rm dw}\widetilde{\mathcal{X}} \cap (\mathcal{E}\text{-}{\rm dw}\widetilde{\mathcal{X}})^{\perp} = \widetilde{\mathcal{X}}_\mathcal{E} \subseteq \mathcal{E}\text{-}{\rm dg}\widetilde{\mathcal{X}}$ and $X'\in \mathcal{E}\text{-}{\rm dw}\widetilde{\mathcal{X}} $.
Thus we have the following diagram in ${\rm Ch}(\mathcal{A}, \mathcal{E})$ with a morphism $g: M\rightarrow D$ such that $g\beta = f$ by ${\rm Ext}^1_{{\rm Ch}(\mathcal{E})}(X', D) = 0$:
$$\xymatrix{ & &X \ \ar@{>->}[r]^{\beta}\ar[d]_{f} &M\ar@{-->}[ld]^g\ar@{-->}[lld]\ar@{->>}[r]  &X'\\
K\ \ar@{>->}[r] &C\ar@{->>}[r]^{\alpha} & D. & & }$$
Note that $K\in \widetilde{\mathcal{E}}$, it follows from Proposition \ref{thm:(dgX,E)} that ${\rm Ext}^1_{{\rm Ch}(\mathcal{E})}(M, K) = 0$, hence we get a morphism $h: M\rightarrow C$ such that $\alpha h = g$. Then we have $\alpha h \beta =g \beta =f$ which proves the above claim.

Then, the model structure $\mathcal{M}_{dw\mathcal{X}}$ follows immediately from \cite[Theorem 3.3]{Gil11}, Lemma \ref{lem:2Xe}, Propositions \ref{prop:(dwX,+)} and \ref{prop:(X,Ch)}. It is clear that ${\rm Ho}(\mathcal{M}_{dw\mathcal{X}}) \simeq {\bf K}(\mathcal{X})$. Dually, we obtain the model structure  $\mathcal{M}_{dw\mathcal{Y}}$ and  ${\rm Ho}(\mathcal{M}_{dw\mathcal{Y}}) \simeq {\bf K}(\mathcal{Y})$.
\end{proof}

\begin{theorem}\label{thm:KX=KY}
Assume that $(\mathcal{E}\text{-}{\rm dw}\widetilde{\mathcal{X}})^{\perp}= {}^{\perp}(\mathcal{E}\text{-}{\rm dw}\widetilde{\mathcal{Y}})$. Then there is a triangle-equivalence $${\bf K}(\mathcal{X}) \simeq {\bf K}(\mathcal{Y}).$$
\end{theorem}

\begin{proof}
The model structures $\mathcal{M}_{dw\mathcal{X}}$ and $\mathcal{M}_{dw\mathcal{Y}}$ come from Theorem \ref{thm:M5}. Since $(\mathcal{E}\text{-}{\rm dw}\widetilde{\mathcal{X}})^{\perp}= {}^{\perp}(\mathcal{E}\text{-}{\rm dw}\widetilde{\mathcal{Y}})$, they share the same trivial objects. Then we get equivalences of the triangulated categories $${\bf K}(\mathcal{X}) \simeq {\rm Ho}(\mathcal{M}_{dw\mathcal{X}})\simeq {\rm Ho}(\mathcal{M}_{dw\mathcal{Y}})\simeq {\bf K}(\mathcal{Y})$$
by \cite[Corollary 1.4]{GLZ24}.
\end{proof}

We show that the next result recovers \cite[Theorem A]{Chen10} by Quillen equivalences.

\begin{corollary}
If $\mathcal{X}\text{-}{\rm dim}\mathcal{A} ( = \mathcal{Y}\text{-}{\rm codim}\mathcal{A}) <\infty$, then ${\bf K}(\mathcal{X})\simeq {\bf K}(\mathcal{Y})$.
\end{corollary}

\begin{proof}
It follows from Remark \ref{rem:1} that if $\mathcal{X}\text{-}{\rm dim}\mathcal{A} <\infty$, then $\mathcal{E}\text{-}{\rm dg}\widetilde{\mathcal{X}} = \mathcal{E}\text{-}{\rm dw}\widetilde{\mathcal{X}}$. Dually, we have $\mathcal{E}\text{-}{\rm dg}\widetilde{\mathcal{Y}} = \mathcal{E}\text{-}{\rm dw}\widetilde{\mathcal{Y}}$. The remaining assertions are immediate from Theorem \ref{thm:KX=KY}, Proposition \ref{thm:(dgX,E)} and its dual.
\end{proof}

\section{Applications}

\subsection{Cotorsion triples}

Throughout this subsection, we suppose that $\mathcal{A}$ is a complete and cocomplete abelian category and that $(\mathcal{X}, \mathcal{Z}, \mathcal{Y})$ is a complete and hereditary cotorsion triple in $\mathcal{A}$. This implies that $(\mathcal{X}, \mathcal{Y})$ is an admissible balanced pair in $\mathcal{A}$. Denote by $\mathcal{P}$ and $\mathcal{I}$ the subcategories of projective and injective objects of $\mathcal{A}$, respectively.
By \cite[Theorem 4.4 and Proposition 3.2]{EPZ20}, $\mathcal{A}$ has enough projectives and injectives such that $\mathcal{X}\cap\mathcal{Z} = \mathcal{P}$ and $\mathcal{Z}\cap\mathcal{Y} = \mathcal{I}$.

Let $\widetilde{\mathcal{Z}}$ be the class of acyclic complexes $Z$ with all ${\rm Z}_nZ \in \mathcal{Z}$. Denote by ${\rm dg}\widetilde{\mathcal{X}}$ the class of complexes $X\in \mathcal{E}\text{-}{\rm dw}\widetilde{\mathcal{X}}$ for which every map $X\rightarrow Z$ is null homotopic whenever $Z \in \widetilde{\mathcal{Z}}$, and by ${\rm dg}\widetilde{\mathcal{Y}}$ the class of complexes $Y\in \mathcal{E}\text{-}{\rm dw}\widetilde{\mathcal{Y}}$ for which every map $Z\rightarrow Y$ is null homotopic whenever $Z \in \widetilde{\mathcal{Z}}$. As shown by Gillespie \cite[Corollary 3.13]{Gil04}, both $({\rm dg}\widetilde{\mathcal{X}}, \widetilde{\mathcal{Z}})$ and $(\widetilde{\mathcal{Z}}, {\rm dg}\widetilde{\mathcal{Y}})$ are hereditary cotorsion pairs in ${\rm Ch}(\mathcal{A})$.

It follows from the definition that $\widetilde{\mathcal{Z}}\subseteq \widetilde{\mathcal{E}}$, thus $\mathcal{E}\text{-}{\rm dg}\widetilde{\mathcal{X}}\subseteq {\rm dg}\widetilde{\mathcal{X}}$. For applications, we would point out the following observations:

\begin{lemma}
Any complex in $\widetilde{\mathcal{E}} \cap {\rm dg}\widetilde{\mathcal{X}}$ is contractible.
\end{lemma}

\begin{proof}
For a complex $X$ in $\widetilde{\mathcal{E}} \cap {\rm dg}\widetilde{\mathcal{X}}$, $X$ is acyclic, thus by \cite[Theorem 3.12]{Gil04}, $X \in \widetilde{\mathcal{X}}$. Hence for any $n$, ${\rm Z}_nX\in \mathcal{X}$ and the right $\mathcal{X}$-acyclic sequence
$$0\rightarrow {\rm Z}_nX\rightarrow X\rightarrow {\rm Z}_{n-1}X\rightarrow 0$$
splits, as required.
\end{proof}

\begin{lemma}\label{dg=dwthenedg=edw}
${\rm dg}\widetilde{\mathcal{X}} = \mathcal{E}\text{-}{\rm dw}\widetilde{\mathcal{X}}$ if and only if $\mathcal{E}\text{-}{\rm dg}\widetilde{\mathcal{X}} = \mathcal{E}\text{-}{\rm dw}\widetilde{\mathcal{X}}$.
\end{lemma}

\begin{proof}
It suffices to prove that $\mathcal{E}\text{-}{\rm dw}\widetilde{\mathcal{X}}\subseteq \mathcal{E}\text{-}{\rm dg}\widetilde{\mathcal{X}}$ in the case of ${\rm dg}\widetilde{\mathcal{X}} = \mathcal{E}\text{-}{\rm dw}\widetilde{\mathcal{X}}$. For any $X\in \mathcal{E}\text{-}{\rm dw}\widetilde{\mathcal{X}} = {\rm dg}\widetilde{\mathcal{X}}$, by Proposition \ref{thm:(dgX,E)} we get a short exact sequence $0\rightarrow K\rightarrow L\rightarrow X\rightarrow 0$ in ${\rm Ch}(\mathcal{A}, \mathcal{E})$ with $L\in \mathcal{E}\text{-}{\rm dg}\widetilde{\mathcal{X}}\subseteq {\rm dg}\widetilde{\mathcal{X}}$ and $K \in \widetilde{\mathcal{E}}$. Since $({\rm dg}\widetilde{\mathcal{X}}, \widetilde{\mathcal{Z}})$ is a hereditary cotorsion pair, $K \in \widetilde{\mathcal{E}} \cap {\rm dg}\widetilde{\mathcal{X}}$ is contractible, which implies $X \cong L$ in ${\rm K}(\mathcal{A})$ and thus in $\mathcal{E}\text{-}{\rm dg}\widetilde{\mathcal{X}}$.
\end{proof}

An analogous dual statement is that $\mathcal{E}\text{-}{\rm dg}\widetilde{\mathcal{Y}} = \mathcal{E}\text{-}{\rm dw}\widetilde{\mathcal{Y}}$ if and only if ${\rm dg}\widetilde{\mathcal{Y}} = \mathcal{E}\text{-}{\rm dw}\widetilde{\mathcal{Y}}$. The following corollary shows that the triangle-equivalence in Theorem \ref{thm:KX=KY} can be restricted to ${\bf K}(\mathcal{P}) \simeq {\bf K}(\mathcal{I})$.

\begin{corollary}\label{cor:KP=KI}If $\mathcal{E}\text{-}{\rm dg}\widetilde{\mathcal{X}} = \mathcal{E}\text{-}{\rm dw}\widetilde{\mathcal{X}}$ and $\mathcal{E}\text{-}{\rm dg}\widetilde{\mathcal{Y}} = \mathcal{E}\text{-}{\rm dw}\widetilde{\mathcal{Y}}$, then we have a triangle-equivalence ${\bf K}(\mathcal{X}) \simeq {\bf K}(\mathcal{Y})$, which restricts to a triangle-equivalence ${\bf K}(\mathcal{P}) \simeq {\bf K}(\mathcal{I})$.
\end{corollary}

\begin{proof} Since $(\mathcal{E}\text{-}{\rm dg}\widetilde{\mathcal{X}})^{\perp} = \widetilde{\mathcal{E}} = {}^{\perp}(\mathcal{E}\text{-}{\rm dg}\widetilde{\mathcal{Y}})$, the hypothesis of Theorem \ref{thm:KX=KY} is satisfied. Then by Theorem \ref{thm:KX=KY} and \cite[Section 1.3]{Hov99}, the triangle-equivalence ${\bf K}(\mathcal{X}) \simeq {\bf K}(\mathcal{Y})$ is implemented by mutually inverse total derived functors $F:{\bf K}(\mathcal{X}) \to {\bf K}(\mathcal{Y})$ and $G:{\bf K}(\mathcal{Y}) \to {\bf K}(\mathcal{X})$. For more details, we refer to \cite[Chapter 7]{G}. Since $\mathcal{X}\cap\mathcal{Z} = \mathcal{P}$ and $\mathcal{Z}\cap\mathcal{Y} = \mathcal{I}$ by \cite[Proposition 3.2]{EPZ20}, it remains to verify that
$F({\bf K}(\mathcal{P}))\subseteq{{\bf K}(\mathcal{I})}$ and $G({\bf K}(\mathcal{I}))\subseteq{{\bf K}(\mathcal{P})}$.

Specifically, the total derived functor $F:{\bf K}(\mathcal{X})\rightarrow{\bf K}(\mathcal{Y})$ is defined by taking a fibrant replacement in the sense of \cite[p.5]{Hov99}, i.e., a special left $\mathcal{E}\text{-}{\rm dw}\widetilde{\mathcal{Y}}$-approximation by using the completeness of cotorsion pair $({}^{\perp}(\mathcal{E}\text{-}{\rm dw}\widetilde{\mathcal{Y}}),\mathcal{E}\text{-}{\rm dw}\widetilde{\mathcal{Y}})$ in ${\rm Ch}(\mathcal{A}, \mathcal{E})$. For any complex $M\in {\bf K}(\mathcal{P})$, by \cite[Theorem 4.22]{SS11} we have a short exact sequence $$0\rightarrow M\stackrel{f}\rightarrow Y\rightarrow L\rightarrow 0$$ with $Y\in {\rm dg}\widetilde{\mathcal{Y}}\subseteq \mathcal{E}\text{-}{\rm dw}\widetilde{\mathcal{Y}}$ and $L\in \widetilde{\mathcal{Z}}\subseteq \widetilde{\mathcal{E}}$ in ${\rm Ch}(\mathcal{A})$. Since each $M_i\in \mathcal{P} \subseteq \mathcal{Z}$ and $(\mathcal{X}, \mathcal{Z})$ is a cotorsion pair in $\mathcal{A}$, $0\rightarrow M\stackrel{f}\rightarrow Y\rightarrow L\rightarrow 0$ is degree-wise right $\mathcal{X}$-acyclic. Observe that $\mathcal{E}\text{-}{\rm dg}\widetilde{\mathcal{Y}} = \mathcal{E}\text{-}{\rm dw}\widetilde{\mathcal{Y}}$. It follows that
$f:M\rightarrow Y$ is a special left $\mathcal{E}\text{-}{\rm dw}\widetilde{\mathcal{Y}}$-approximation of $M$ in ${\rm Ch}(\mathcal{A}, \mathcal{E})$, which implies $F(M)=Y$. Since $L$ is acyclic and its cycles lie in $\mathcal{Z}$, each $L_i$ belongs to $\mathcal{Z}$. Consequently, each $Y_i$ lies in $\mathcal{Z} \cap \mathcal{Y} = \mathcal{I}$, and hence $F(M) \in {\bf K}(\mathcal{I})$, as desired. Dually, consider the total derived functor $G:{\bf K}(\mathcal{Y})\rightarrow{\bf K}(\mathcal{X})$, which takes a cofibrant replacement, i.e., a special right $\mathcal{E}\text{-}{\rm dw}\widetilde{\mathcal{X}}$-approximation by using the completeness of cotorsion pair $(\mathcal{E}\text{-}{\rm dw}\widetilde{\mathcal{X}}, (\mathcal{E}\text{-}{\rm dw}\widetilde{\mathcal{X}})^{\perp})$ in ${\rm Ch}(\mathcal{A}, \mathcal{E})$. The argument that $G$ maps complexes in ${\bf K}(\mathcal{I})$ into ${\bf K}(\mathcal{P})$ is analogous to the one above. This completes the proof.
\end{proof}
%

\subsection{Gorenstein projective and Gorenstein injective modules}
From now on, throughout this subsection, let $R$ be an associative ring with identity. We use ${\rm Mod}(R)$ and ${\rm Ch}(R)$ to denote the category of left $R$-modules and the category of $R$-complexes, respectively. We also denote by $\mathcal{P}$ and $\mathcal{I}$ the subcategories of projective and injective objects in ${\rm Mod}(R)$, respectively.

Recall that $M \in {\rm Mod}(R)$ is {\em Gorenstein projective} if $M \cong {\rm Z}_0 C$ for some complex $C$ of projective modules which remains acyclic after applying ${\rm Hom}_{\mathcal{A}}(-, P)$ for any $P\in \mathcal{P}$. Similarly, {\em Gorenstein injective} modules are defined, see \cite[Definition 4.1]{SWSW08}. We denote by $\mathcal{GP}$ the full subcategory of ${\rm Mod}(R)$ which is formed by all Gorenstein projective modules, and $\mathcal{GI}$ denotes the full subcategory of all Gorenstein injective modules.

Following \cite{DLW23,WE24}, we will call $R$ a \emph{left virtually Gorenstein} ring if $\mathcal{GP}^{\perp} ={^{\perp}\mathcal{GI}}$ in ${\rm Mod}(R)$. It follows from \cite[Proposition 3.16]{WE24} that $R$ is left virtually Gorenstein if and only if $(\mathcal{GP},
{{\mathcal{GP}}^{\perp}={}^{\perp}\mathcal{GI}}, \mathcal{GI})$ is a complete and hereditary cotorsion triple in ${\rm Mod}(R)$.  In this case, $(\mathcal{GP}, \mathcal{GI})$ forms an admissible balanced pair in ${\rm Mod}(R)$.

\begin{corollary}\label{cor:Rmodcase}
If $R$ is a left virtually Gorenstein ring such that $\mathcal{E}\text{-}{\rm dg}\widetilde{\mathcal{GP}} = \mathcal{E}\text{-}{\rm dw}\widetilde{\mathcal{GP}}$ and $\mathcal{E}\text{-}{\rm dg}\widetilde{\mathcal{GI}} = \mathcal{E}\text{-}{\rm dw}\widetilde{\mathcal{GI}}$, then we have a triangle-equivalence ${\bf K}(\mathcal{GP}) \simeq {\bf K}(\mathcal{GI})$, which restricts to a triangle-equivalence ${\bf K}(\mathcal{P}) \simeq {\bf K}(\mathcal{I})$.
\end{corollary}

Recall that $R$ is called \emph{left Gorenstein} \cite{Bel00} provided that $R$ is a ring of finite global Gorenstein dimension. A consequence of Corollary \ref{cor:Rmodcase} yields the following corollary, which has been obtained by Chen in \cite[Theorem B]{Chen10} via a quite different proof.



\begin{corollary}\label{cor:KGP=KGI'}
Let $R$ be a left Gorenstein ring. Then we have a triangle-equivalence ${\bf K}(\mathcal{GP}) \simeq {\bf K}(\mathcal{GI}),$ which restricts to a triangle-equivalence ${\bf K}(\mathcal{P}) \simeq {\bf K}(\mathcal{I})$.
\end{corollary}
\begin{proof} Since $R$ is a left Gorenstein ring, it follows that $R$ is a left virtually Gorenstein ring such that $\mathcal{E}\text{-}{\rm dg}\widetilde{\mathcal{GP}} = \mathcal{E}\text{-}{\rm dw}\widetilde{\mathcal{GP}}$ and $\mathcal{E}\text{-}{\rm dg}\widetilde{\mathcal{GI}} = \mathcal{E}\text{-}{\rm dw}\widetilde{\mathcal{GI}}$. So the result holds by Corollary \ref{cor:Rmodcase}.
\end{proof}

Let $\mathbf{F} = \cdots\rightarrow F_{1}\rightarrow F_{0}\rightarrow F_{-1}\rightarrow\cdots$ be an acyclic complex of left $R$-modules which remains acyclic after applying $I\otimes_{R}-$ for any injective right $R$-module $I$. Then $M \cong \mathrm{Ker}(F_0\rightarrow F_{-1})$ is called a {\em Gorenstein flat module} provided that each $F_i$ is a flat left $R$-module. Recall that the {\em Gorenstein weak dimension} of $R$, denoted by ${\rm G.wdim}R$, is defined as the supremum of Gorenstein flat dimension of all left $R$-modules.

The next interesting observation follows from \cite[Theorem 4.2]{WE24}, which provides another sufficient condition on the base ring $R$ such that $(\mathcal{GP}, \mathcal{GI})$ is a balanced pair. We refer to \cite{DLW23} for a more detailed discussion on this matter.

\begin{lemma}\label{lem:G-Tri1}
Let $R$ be a ring. Then ${\rm G.wdim}R < \infty$ if and only if $(\mathcal{GP}, \mathcal{V}, \mathcal{GI})$ is a complete hereditary cotorsion triple, where $\mathcal{V}$ is the class of modules with finite flat dimension.
\end{lemma}

Let $\mathcal{E}$ be the class of short exact sequences of complexes which remain exact after applying ${\rm Hom}_R(\mathcal{GP},-)$. Denote by ${\rm Ch}(R, \mathcal{E})$ the exact category of $R$-complexes with respect to the short exact sequences which are in $\mathcal{E}$ in each degree. The next observation is from \cite[Lemmas 4.5 and 4.6]{WE24}, together with Lemma \ref{dg=dwthenedg=edw} and its duality:

\begin{lemma}\label{lem:G-Tri2}
Let $R$ be a ring with ${\rm G.wdim}R < \infty$. Then $\mathcal{E}\text{-}{\rm dg}\widetilde{\mathcal{GP}} = \mathcal{E}\text{-}{\rm dw}\widetilde{\mathcal{GP}}$ and $\mathcal{E}\text{-}{\rm dg}\widetilde{\mathcal{GI}} = \mathcal{E}\text{-}{\rm dw}\widetilde{\mathcal{GI}}$.
\end{lemma}

Then by Corollary \ref{cor:Rmodcase}, together with Lemmas \ref{lem:G-Tri1} and \ref{lem:G-Tri2}, we get the result stated in \cite[Theorem 1.2]{WE24} via a quite different method.

\begin{corollary}\label{cor:KGP=KGI} {\rm(\cite[Theorem 1.2]{WE24})}
Let $R$ be a ring with ${\rm G.wdim}R < \infty$. Then we have a triangle-equivalence ${\bf K}(\mathcal{GP}) \simeq {\bf K}(\mathcal{GI}),$ which restricts to a triangle-equivalence ${\bf K}(\mathcal{P}) \simeq {\bf K}(\mathcal{I})$.
\end{corollary}

\subsection{Pure projective and pure injective objects}

Throughout this subsection, let $\mathcal{A}$ be a locally finitely presented Grothendieck category with an admissible balanced pair $(\mathcal{X}, \mathcal{Y})$.

Recall that an acyclic complex in $\mathcal{A}$ is {\em pure acyclic} provided that it remains acyclic by applying $\mathrm{Hom}_{\mathcal{A}}(P,-)$ for any finitely presented object $P$. A short exact sequence
$$0\rightarrow M_1\rightarrow M_2\rightarrow M_3\rightarrow 0$$
is said to be {\em pure exact} if it is pure acyclic as a complex, and we denote by $\mathcal{E}$ the class of pure exact sequences in $\mathcal{A}$.

An object $P$ is called {\em pure projective} if $\mathrm{Hom}_{R}(P,-)$ is an exact functor on all pure exact sequences. Thus projective objects and finitely presented objects are pure projective. We use $\mathcal{PP}$ to denote the class of all pure projective objects in $\mathcal{A}$. Dually, one has the notion of pure injective objects; the class of pure injective objects in $\mathcal{A}$ is denoted by $\mathcal{PI}$. It follows that $(\mathcal{PP}, \mathcal{PI})$ is an admissible balanced pair and the proof is similar to that of \cite[Example 8.3.2]{EJ00}.

The following results give some nice characterizations of the right and left orthogonal classes of complexes of pure projective and injective objects, respectively.

\begin{proposition} {\rm (\cite[Theorem 5.4]{S14})}\label{Thm5.4ofS14}
Denote by $\mathcal{W}$ the classes of pure acyclic complexes, there exists complete and hereditary cotorsion pairs $(\mathcal{E}\text{-}{\rm dw}\widetilde{\mathcal{PP}}, \mathcal{W})$ and $(\mathcal{W}, \mathcal{E}\text{-}{\rm dw}\widetilde{\mathcal{PI}})$ in ${\rm Ch}(\mathcal{A}, \mathcal{E})$.
\end{proposition}

The following result is a direct consequence of Proposition \ref{Thm5.4ofS14}.

\begin{lemma}\label{orthogonalclassesequal}
Let $\mathcal{A}$ be a locally finitely presented Grothendieck category. Then $(\mathcal{E}\text{-}{\rm dw}\widetilde{\mathcal{PP}})^{\perp} = {}^{\perp}(\mathcal{E}\text{-}{\rm dw}\widetilde{\mathcal{PI}})$ holds on ${\rm Ch}(\mathcal{A}, \mathcal{E})$.
\end{lemma}

Thus, we obtain the following corollary from Theorem \ref{thm:KX=KY} and Lemma \ref{orthogonalclassesequal}.

\begin{corollary}\label{cor:KPP=KPI}
Let $\mathcal{A}$ be a locally finitely presented Grothendieck category. Then there is a triangle-equivalence ${\bf K}(\mathcal{PP}) \simeq {\bf K}(\mathcal{PI})$.
\end{corollary}

\begin{remark}\label{remark:5.11}
The authors are grateful to a referee of this paper for noting that Corollary \ref{cor:KPP=KPI} has already been proved by \v{S}\v{t}ov\'{\i}\v{c}ek in \cite[Corollaries 5.7 and 5.8]{S14}. Furthermore, Corollary \ref{cor:KPP=KPI} may be viewed as a refinement of a result by Chen in \cite{Chen10}, who obtained the triangle-equivalence ${\bf K}(\mathcal{PP}) \simeq {\bf K}(\mathcal{PI})$ under the additional hypothesis that $R$ is a ring with finite pure global dimension. It is also worth noting that the proof presented here is distinct from the one given in \cite{Chen10}.
\end{remark}

\begin{remark}
It is tempting to generalize the application presented in Section~5.3 from pure projective and pure injective objects to the broader framework of $\lambda$-pure projective and $\lambda$-pure injective objects, where $\lambda$ denotes an infinite regular cardinal. However, for uncountable $\lambda$, the existence of sufficiently many $\lambda$-pure injective preenvelopes is a delicate issue and is not guaranteed in general. We are grateful to the referee for bringing this to our attention and for supplying the relevant reference; further details can be found in~\cite{CS25}.
\end{remark}

\vskip 10pt

\noindent {\bf Acknowledgements.}\quad J.S. Hu and X.Y. Yang are supported by the National Natural Science Foundation of China (Grant No. 12571035). W. Ren is supported by the Natural Science Foundation of Chongqing, China (No. CSTB2025NSCQ-GPX1014). H.Y. You is supported by Zhejiang Provincial Natural Science Foundation of China (No. LQ23A010004) and the National Natural Science Foundation of China (Grant No. 12401043). The authors are grateful to the referees for reading the paper carefully and for many suggestions on mathematics and English expressions.

\newpage

\bibliography{}

\vspace{3mm}
\noindent\textbf{Jiangsheng Hu}\\
School of Mathematics, Hangzhou Normal University, Hangzhou 311121, P. R. China.\\
Email: \textsf{hujs@hznu.edu.cn}\\[1mm]
\textbf{Wei Ren}\\
School of Mathematical Sciences, Chongqing Normal University, Chongqing 401331, P. R. China\\
Email: \textsf{wren@cqnu.edu.cn}\\[1mm]
\textbf{Xiaoyan Yang}\\
School of Science, Zhejiang University of Science and Technology, Hangzhou 310023, P. R. China.\\
Email: \textsf{yangxy@zust.edu.cn}\\[1mm]
\textbf{Hanyang You}\\
School of Mathematics, Hangzhou Normal University, Hangzhou 311121, P. R. China.\\
E-mail: \textsf{youhanyang@hznu.edu.cn}\\[1mm]

%
%
%
%

\end{document}